\documentclass{article}

\usepackage{lineno,hyperref}
\modulolinenumbers[5]
\usepackage{amsfonts}

\usepackage[usenames,dvipsnames]{color}
\usepackage{amsmath}
\usepackage{amsfonts}
\usepackage{calc}
\usepackage{hyperref}

\usepackage{booktabs}
\usepackage{multirow}
\usepackage{rotating} 

\usepackage{graphicx}
\usepackage{caption}
\usepackage{subcaption}
\usepackage{placeins}

\usepackage[margin=1in]{geometry}

\usepackage{amsthm}
\newtheorem{theorem}{Theorem}
\newtheorem{lemma}{Lemma}

\newtheorem{example}{Example}
\newtheorem{definition}{Definition}
\newtheorem{assumption}{Assumption}
\newtheorem{corollary}{Corollary}
\newtheorem{remark}{Remark}

\newcommand{\R}{\mathbb{R}}

\newcommand{\TTS}{\text{TTS}}
\newcommand{\TFT}{\text{TFT}}
\newcommand{\TWT}{\text{TWT}}

\title{Sufficient Optimality Conditions for Distributed, Non-Predictive Ramp Metering in the Monotonic Cell Transmission Model\thanks{Research was supported by the European Union 7th Framework Program ``Scalable Proactive Event-Driven Decision-making (SPEEDD)" (FP7-ICT 619435). }
}

\author{Marius Schmitt, Chithrupa Ramesh, and John Lygeros
\thanks{$^{1}$Automatic Control Laboratory, ETH Zurich, Switzerland,     {\tt\small <schmittm,rameshc,lygeros>@control.ee.ethz.ch}}%
}

\begin{document}

\maketitle

\begin{abstract}
We consider the ramp metering problem for a freeway stretch modeled by the Cell Transmission Model. Assuming perfect model knowledge and perfect traffic demand prediction, the ramp metering problem can be cast as a finite horizon optimal control problem with the objective of minimizing the Total Time Spent, i.e., the sum of the travel times of all drivers. For this reason, the application of Model Predictive Control (MPC) to the ramp metering problem has been proposed. However, practical tests on freeways show that MPC may not outperform simple, distributed feedback policies. Until now, a theoretical justification for this empirical observation was lacking.  This work compares the performance of distributed, non-predictive policies to the optimal solution in an idealised setting, specifically, for monotonic traffic dynamics and assuming perfect model knowledge. To do so, we suggest a distributed, non-predictive policy and derive sufficient optimality conditions for the minimization of the Total Time Spent via monotonicity arguments. In a case study based on real-world traffic data, we demonstrate that these optimality conditions are only rarely violated. Moreover, we observe that the suboptimality resulting from such infrequent violations appears to be negligible. We complement this analysis with simulations in non-ideal settings, in particular allowing for model mismatch, and argue that Alinea, a successful, distributed ramp metering policy, comes close to the ideal controller both in terms of control behavior and in performance.
\end{abstract}




\section{Introduction} \label{sec:introduction} 

Ramp metering refers to the active control of the inflow of cars on a freeway via the onramps, by means of installing and controlling a traffic light at every onramp. In this work, we consider the freeway ramp metering problem over a finite horizon, e.g.\ one day or one rush-hour period. Freeway ramp metering has been established as an effective and practically useful tool to improve traffic flows on congestion-prone freeways \cite{papageorgiou2003review,papageorgiou2000freeway}. We study the problem by adopting the Cell Transmission Model (CTM) for freeways, as introduced in the seminal work by Daganzo \cite{daganzo1994cell,daganzo1995cell}, which can be interpreted as a first-order Godunov approximation of the continuous Lighthill-Whitham-Richards-model (LWR) \cite{lighthill1955kinematic,richards1956shock}. Modifications and generalizations of this model have since been introduced, and we consider a slight generalization of the original, piece-wise affine CTM to a more general, concave fundamental diagram \cite{como2016convexity,coogan2016stability,lovisari2014stability2}. Furthermore, we employ a simplified onramp model originally introduced as the ``asymmetric" CTM \cite{gomes2006optimal,gomes2008behavior}, which simplifies the model of onramp-mainline merges by distinguishing mainline- and onramp-traffic demand.

The popularity of the CTM for model-based control stems from the simplicity of the model equations, allowing for computationally efficient solution methods for finite horizon optimal control problems. In particular, relaxing the piecewise-affine fundamental diagram allows one to pose such optimal control problems as linear programs \cite{ziliaskopoulos2000linear}. The solution of the relaxed problem can be made feasible for the original, non-relaxed problem by employing mainline demand control \cite{muralidharan2012optimal}. In addition, conditions on the structure of the road network and its dynamics have been derived, which ensure that ramp metering alone is sufficient to make solutions feasible \cite{gomes2006optimal,como2016convexity}.

A conceptually different approach uses distributed feedback. In such ramp metering policies, local controllers only receive measurements from sensors in close vicinity to any particular onramp and only exchange limited amounts of information, if at all \cite{papageorgiou1991alinea,stephanedes1994implementation,zhang1997freeway}. Those control policies aim at maximizing bottleneck flows locally, but have been shown to come close to the performance of optimal ramp metering policies in real-world evaluations \cite{smaragdis2004flow,wang2013local,papamichail2010heuristic}. While it is apparent that such local feedback controllers are far easier to implement than model-based optimal control policies, it is not obvious why and when the performance of distributed, non-predictive ramp metering policies comes close to the centralized, optimal control solution. A special case for which the optimal control policy can be explicitly constructed is analyzed in \cite{zhang2004optimal}. It is stated that the structure of the explicit solution ``explains why some local metering algorithms [...] are successful -- they are really close to the most-efficient logic". However, no proof of optimality is provided. An interesting result exists for the problem of controlling a network of signalized intersections, modeled as a queue network. In particular, throughput optimality of a distributed controller, the so-called max-pressure controller, has been shown via network calculus arguments \cite{varaiya2013max}. However, the employed network model does not include congestion spill-back effects.

In addition, distributed policies are popular in routing problems, where the route choice through a network is (partially) controlled. It was established that monotonicity of certain routing policies can be leveraged in the analysis of its robustness with respect to non-anticipated reductions of the capacity of individual links \cite{como2013robust1,como2013robust2}. Subsequently, a generalized class of monotonic, distributed routing policies which implicitly back-propagate congestion effects was introduced and it was proven that such policies stabilize an equilibrium that maximizes throughput \cite{como2015throughput}, for dynamical network models reminiscent of the CTM. Even though we will not assume that split ratios can be actuated in this work, it is interesting to note that monotonicity of parts of the system dynamics will also be essential to our analysis.

This work addresses the question of how distributed, non-predictive ramp metering policies compare to optimization-based, centralized and predictive control policies, in terms of minimization of the Total Time Spent (TTS) for a freeway stretch modeled by the monotonic CTM. If perfect model knowledge and perfect traffic demand prediction is assumed, the ramp metering problem can be cast as a finite horizon optimal control problem, which can be reformulated as a Linear Program in case of a piecewise-affine fundamental diagram \cite{gomes2006optimal,como2016convexity}. However, both perfect traffic demand prediction and perfect model knowledge are unrealistic \cite{burger2013considerations}. We retain the assumption of perfect model knowledge for the purpose of the theoretical analysis but drop the assumption of prediction of external demands and introduce a distributed, non-anticipative feedback controller. This controller can be motivated as a one-step-ahead maximization of local traffic flows and is hence called the best-effort controller. We proceed to derive conditions under which such a controller performs optimally, that is, minimizes TTS. To demonstrate the applicability of our results, we perform a simulation case study based on a real-world freeway described in \cite{de2015grenoble}, using the monotonic CTM with freeway parameters and demand profiles estimated from real measurements. The simulations confirm the theoretical results as can be seen on days when the optimality conditions are satisfied at all time steps. However, our results do not provide any a priori bound on the suboptimality of the solution even for small violations of the conditions. 
Nevertheless, even on days during which violations of the optimality conditions are observed, a-posteriori suboptimality bounds can be computed by employing results from the theoretical analysis. In the evaluation, we find that violations of the optimality conditions affect only a small number of time steps for any given day. We also observe that such infrequent violations tend to lead to negligible optimality gaps, although this can only be verified a-posteriori.

The main contribution of this work lies in the derivation of sufficient optimality conditions for minimal TTS ramp metering in the monotonic CTM, which provide a theoretical explanation of why and when the performance of distributed, non-predictive ramp metering and optimal control policies coincides. The assumption of perfect model knowledge employed for the sake of the theoretical analysis is not practical, of course, but we proceed to show that the best-effort controller can in turn be interpreted as an idealized version of the successful Alinea controller \cite{papageorgiou1991alinea}, which replaces the need for perfect model knowledge by virtue of feedback.

Note that the analysis in this work is applicable to the monotonic CTM as studied in e.g.\ \cite{gomes2006optimal,como2016convexity}. Empirical evidence suggests the existence of a non-monotonic capacity drop of freeway traffic flow in congestion and (first-order) models incorporating this effect have been proposed \cite{kontorinaki2016capacity,yuan2017kinematic}. However, so far no general, efficient solution method (e.g. based on a convex reformulation) of optimal control problems for the non-monotonic CTM is known and this work does not directly generalize to a non-monotonic setting. Rather, it should be viewed as an intermediate step towards solving the more general problem.

The paper is organized as follows:  Section \ref{sec:problem_formulation} defines the CTM and the main problem of minimizing TTS. In Section \ref{sec:local_feedback}, we introduce the non-predictive, distributed best-effort controller. We also show that this controller is not always optimal, by stating two counterexamples. In Section \ref{sec:optimality_conditions}, we derive sufficient optimality conditions that can be used to check optimality a posteriori. Section \ref{sec:application} introduces a case study based on real-world freeway parameters, and we demonstrate that the optimality gap of best-effort control is negligibly small. Furthermore, we compare best-effort control to Alinea. Section \ref{sec:conclusions} concludes and discusses implications for the development of real-world, predictive, coordinated ramp-metering policies.

We adopt the following notation: The k-th component of a vector $x$ is denoted $x_k$. The index $k$ will always be used to denote the index of a cell in the CTM and hence $k \in \{1,\dots,n\}$ with $n$ the number of cells in the model, unless a different range is explicitly specified. If a variable $x$ evolves over time, the value at time $t$ is denoted by $x(t)$. We consider discrete time models over a horizon $T$, i.e.,\ $t \in \{0,\dots, T\}$. The operators $\leq, \geq, <, >$ used with vectors denote elementwise inequalities. Also, we write $\left[ x \right]^{u}_{l} := \min \left\{ u, \max \left\{ x, l \right\} \right\}$ for a saturation at an upper bound $u$ and lower bound $l$.

\section{Problem formulation} \label{sec:problem_formulation} 

\begin{figure}[t] 
	\centering
		\includegraphics[width=0.9\columnwidth]{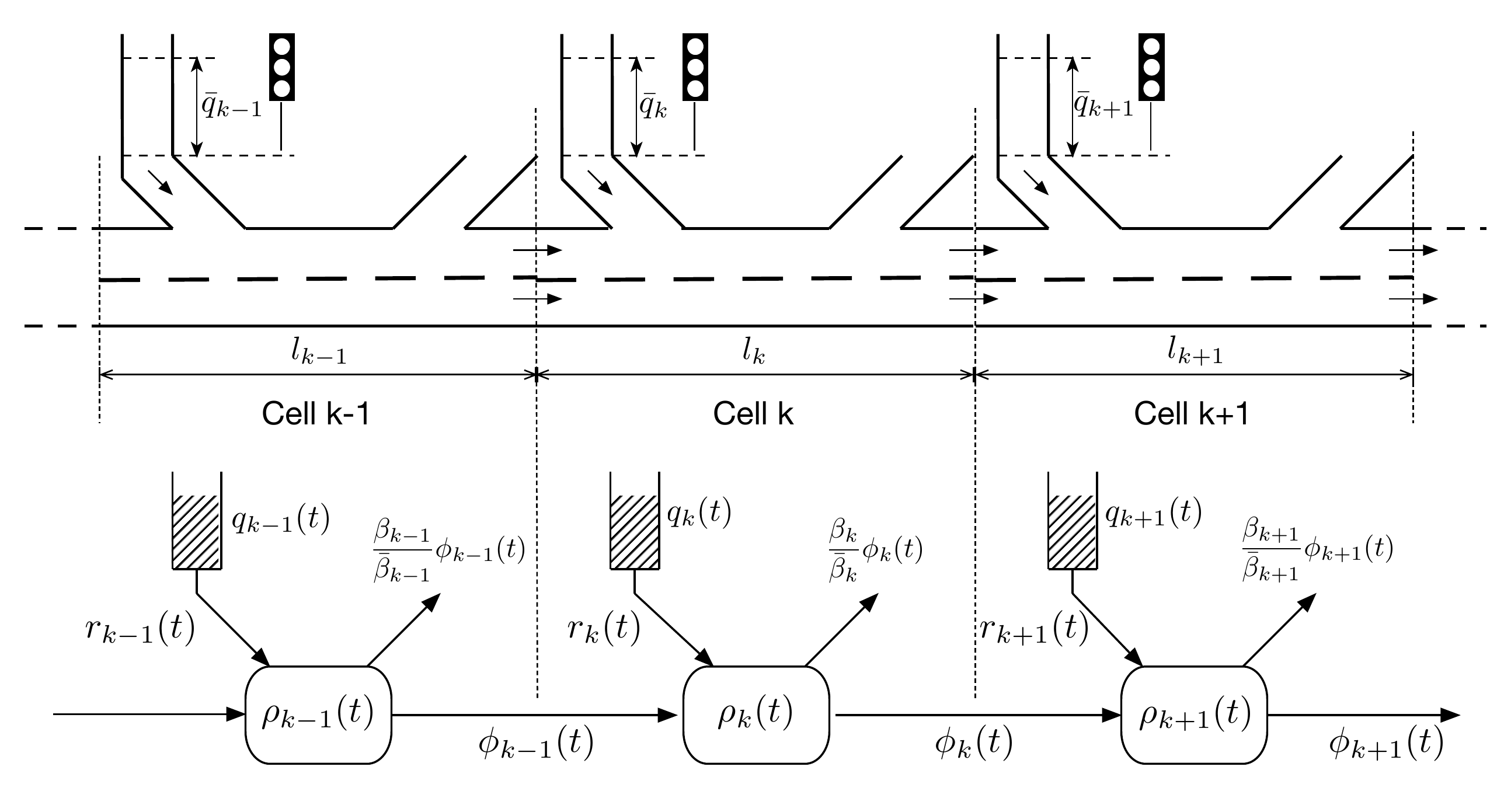}
	\caption{Sketch of the CTM of a freeway segment controlled by ramp metering}
	\label{fig:ctm_sketch}
\end{figure}

We consider the (asymmetric) CTM as introduced in \cite{gomes2006optimal,gomes2008behavior}, which employs a simplified onramp model in comparison to the original model \cite{daganzo1994cell,daganzo1995cell}. In the CTM, the freeway is partitioned into \textit{cells} of length $l_k$, as visualized in Figure \ref{fig:ctm_sketch}. The state of the mainline is described by the traffic \textit{density} $\rho_k(t)$, i.e.,\ the number of cars per length in each cell. The metered onramps are modeled as integrators, and their state is given by the queue length $q_k(t)$, i.e.,\ the number of cars waiting in the queue. The evolution of the system is described by flows of cars during discrete time intervals of duration $\Delta t$. The mainline \emph{flow} between cells $k$ and $k+1$ in time interval $t$ is denoted by $\phi_k(t)$. We make the assumption of constant split ratios, which means that the outflows from the offramps are modeled as percentages $\beta_k$ of the mainline flows, where $\beta_k$ are designated as the \emph{split ratios} of traffic at the offramp in cell $k$. For notational convenience, we also use $\bar\beta_k := 1 - \beta_k$. For metered onramps, the inflows to the freeway are given by the ramp  \textit{metering rates} $r_k(t)$.  The \emph{external traffic demands} $w_k(t)$, i.e.,\ the number of cars per unit time seeking to enter the freeway from either a ramp $k \in \{1,\dots,n\}$ or from the upstream mainline, act as external disturbances on the system. The evolution of the state of the CTM is described by the conservation laws
\begin{subequations}
\begin{align}
\rho_k(t+1) &= \rho_k(t) + \frac{\Delta t}{l_k} \left( \phi_{k-1}(t) + r_k(t) - \frac{1}{\bar\beta_k} \phi_k(t) \right) , \label{eq:density_dynamics} \\
q_k(t+1) &= q_k(t) + \Delta t \cdot \left( w_k(t) - r_k(t) \right) . \label{eq:onramp}
\end{align}
\end{subequations}
This onramp model relies on the implicit assumption that congestion does not spill back onto the onramps. In particular, the model assumes that all cars assigned to enter the freeway in a sampling interval can indeed do so. While this assumption might not be satisfied for an uncontrolled freeway, it was shown to be satisfied for a freeway controlled by ramp metering in a case study \cite{gomes2006optimal}. 
Naturally, the ramp metering rates are nonnegative and subject to a constant upper bound
\begin{align}
0 \leq r_k(t) \leq \bar r_k ,
\label{eq:rampbounds1}
\end{align}
 which characterizes the maximal onramp flow. We will assume that the external demand is bounded $0 \leq w_k(t) \leq \bar r_k$ for all $k$ and for all times $t$. In addition, the limited space on the onramp $0 \leq q_k(t) \leq \bar q_k(t) $ potentially limits the allowed ramp metering rates further and using \eqref{eq:onramp}, we obtain bounds
 \begin{align}
\frac{1}{\Delta t} \left(q_k(t) - \bar q_k \right) + w_k(t) \leq r_k(t) \leq \frac{1}{\Delta t} q_k(t) + w_k(t)
\label{eq:rampbounds2}
\end{align}
 in terms of the metering rates. The upper bound simply ensures that only cars seeking admission in time interval $t$ can be admitted to enter the freeway, while the lower bound mandates that the ramp is operated such that the queue length on the onramp never exceeds the length of the ramp, to avoid queue spill back\footnote{The effects of congestion spill back from the onramps to the arterials is usually considered to have much worse effects than a congestion on the freeway mainline. Thus, avoiding such a spill back takes priority over avoiding a congestion on the mainline.}. We do not impose similar bounds on the outflows via the offramps: Effects like a spill back of congestion from the adjacent aterials are beyond the scope of this model and a constant, upper bound on the outflow from an offramp is redundant under the assumption of constant split ratios.

The CTM is a first order model and therefore, the mainline traffic flows $\phi(t)$ are not states but functions of the densities. The mainline flow $\phi_k(t)$ depends on the traffic \textit{demand} $d_k(\rho_k(t))$ proportional to the number of cars that seek to travel downstream from cell $k$ at time $t$, and the \textit{supply} $s_{k+1}(\rho_{k+1}(t))$ of free space available downstream in cell $k+1$ at time $t$. Demand and supply functions are often identified as the \textit{fundamental diagram} of a cell, as depicted in Figure \ref{fig:fd}. In the original work of  \cite{daganzo1994cell,daganzo1995cell}, a piecewise-affine (PWA) fundamental diagram was assumed, as it is obtained from the Godunov discretization of the LWR-model. In practice, one might want to consider more general shapes of the fundamental diagram and hence of the demand- and supply functions, in order to better approximate real world data (see e.g.\ \cite{karafyllis2015global,como2016convexity,coogan2016stability,lovisari2014stability2} for recent examples). In the remainder of this work, we will assume that:

\begin{assumption} 
\label{assumption:bounded_slope}
For every cell $k$, define a maximal density $\bar\rho_k$, called the traffic jam density, and the critical density $\rho_k^c$,  $0< \rho_k^c < \bar\rho_k$. The demand $d_k(x)$, $d_k: [0,\bar \rho] \rightarrow \mathbb{R}_+$ is a Lipschitz-continuous, nondecreasing function with $d_k(0) = 0$ and $d_k(\rho_k^c) = d_k(\bar\rho_k)$. Conversely, the supply $s_k(x)$, $s_k: [0,\bar \rho] \rightarrow \mathbb{R}_+$ is a Lipschitz-continuous, nonincreasing function with $s_k(\bar \rho_k) = 0$ and $s_k(\rho_k^c) = s_k(0)$. Furthermore, the sampling time $\Delta t$ is chosen such that it satisfies the bounds 
\begin{subequations}
\begin{align}
\Delta t \cdot  c^d_k \leq l_k \bar\beta_k  \quad \forall k ,  \\
\Delta t \cdot c^s_k \leq l_k \quad \forall k \label{eq:supply} , 
\end{align}
\end{subequations}
 with respect to the Lipschitz constants of the demand $c^d_k$ and of the supply $c^s_k$.
\end{assumption}
The flow
\begin{subequations}
\label{eq:flow_dynamics}
\begin{align}
\phi_0(t) &= w_0(t), \label{eq:beginflow} \\
\phi_k(t) &= \min\{d_k(\rho_k(t)), s_{k+1}(\rho_{k+1}(t))\} \quad \forall k \in \{1,\dots,n-1\} ,  \label{eq:mainflow} \\ 
\phi_{n}(t) &= d_n(\rho_n(t)) , \label{eq:endflow}
\end{align}
\end{subequations}
 is now given as the minimum of upstream demand and downstream supply. Note that the assumption of monotonicity of the demand functions implies that $ d_k(\rho_k) = d_k(\rho_k^c)$ for all $\rho_k: \rho_k^c \leq \rho_k \leq \bar\rho_k$, i.e.,\ the traffic demand is constant for densities larger than the critical density. Similarly, the assumption on monotonicity of the supply function implies that the supply is constant for all densities smaller than the critical density, i.e., $ s_k(\rho_k) = s_k(\rho_k^c)$ for all $\rho_k : 0 \leq \rho_k \leq \rho_k^c$. It is sometimes helpful to introduce the \emph{capacity} $F_k := \min\{d_k(\rho_k^c), s_{k+1}(\rho_{k+1}^c)\}$ as an additional parameter in the flow equations. The flow equation \eqref{eq:mainflow} can then equivalently be written as
\begin{align*}
\phi_k(t) = \min\{d_k(\rho_k(t)), F_k, s_{k+1}(\rho_{k+1}(t))\} .
\end{align*}
We will use $F_k$ in the flow equations only when it simplifies the presentation of a proof or of an example. The flow equations are slightly modified for the first and the last cell. For the last cell, we assume that the outflow from the freeway is unobstructed, i.e.,\ no congestion exists downstream of the (part of the) freeway that is modeled. Similarly, the inflow in the first cell is given by the external traffic demand $w_0(t)$, which is \emph{not} restricted by the supply in the first cell. The reason for this modification is that the flow into the first cell arises from external traffic demand. The equations therefore ensure that all external demand eventually will be served. By contrast, limiting the inflow according to the supply of free space, as it is done for any internal flow, would amount to discarding the surplus external demand. It is worth emphasizing that local traffic flows \eqref{eq:flow_dynamics} are maximized at the critical density. In particular, the flow $\phi_k$ into some cell $k+1$ with density equal to or smaller than the critical density $\rho_{k+1}(t) \leq \rho_{k+1}^c$ is not constrained by the supply of free space and hence $\phi_{k}(t) = \min\{d_{k}(\rho_{k}(t)), F_k\}$. Similarly, the flow out of a cell $k$ with density equal to or greater than the critical density $\rho_k(t) \geq \rho_k^c$ is not constrained by the traffic demand and hence $\phi_k(t) = \min\{F_k, s_{k+1}(\rho_{k+1}(t))\}$. A CTM freeway model that satisfies Assumption \ref{assumption:bounded_slope} will be called a \emph{monotonic CTM}. The classical, piecewise affine fundamental diagram satisfies this assumption: Here, we have that $d_k(\rho_k) = \min \left\{ \bar\beta_k v_k \rho_k, \frac{w_k}{\bar\beta_k v_k + w_k} \bar\rho_k \right\}$ and $s_k(\rho_k) = \min \left\{ \frac{w_k}{\bar\beta_k v_k + w_k} \bar\rho_k , (\bar\rho_k - \rho_k) w_k \right\}$ with $v_k$ the free-flow speed and $w_k$ the congestion wave speed. Monotonicity assumptions are trivially satisfied for affine demand and supply functions and the condition on $\Delta t$ can be recognized as the stability condition $v_k \cdot \Delta t \leq l_k ,~ \forall k$ that arises if the CTM is derived as a discretization of the wave PDE using the Godunov scheme. Note that in practice, the congestion-wave speed $w_k$ is significantly lower than the free-flow speed, thus the upper bound in inequality \ref{eq:supply} in Assumption \ref{assumption:bounded_slope} is not restrictive. Different shapes of the fundamental diagram encountered in practice are depicted in Figure \ref{fig:fd} for illustration.

\begin{remark}
Assumption \ref{assumption:bounded_slope} is similar to the assumptions on demand and supply functions made in \cite{lovisari2014stability2,coogan2015compartmental,coogan2016stability}, though the latter assumes that demand and supply functions are strictly increasing resp.\ decreasing. In particular, concavity of demand and supply functions is \emph{not} required, and indeed, the two latter references depict an example of a non-concave but monotonic supply function. In addition, all of these references focus on stability and none of them assume that demand and supply functions are jointly maximized at the critical density, an assumption that is critical for our analysis. Such a condition can also be imposed indirectly via the capacity $F_k$, if chosen appropriately. 

Other works concerned with optimal control of the CTM  \cite{gomes2006optimal,muralidharan2012optimal,como2016convexity} assume \emph{in addition} concavity of the demand and supply functions to keep the resulting optimization problems tractable.
Note that the former two references assume the classical triangular fundamental diagram, whose demand and supply functions are both monotonic and concave. We do not need concavity of demand and supply functions for our main result (Theorem \ref{theorem:main}) to hold. However, we will use monotonic and concave (in fact, triangular) fundamental diagrams in the evaluation (Section \ref{sec:application}), in order to be able to compute the optimal finite horizon solution as a benchmark.
\end{remark}

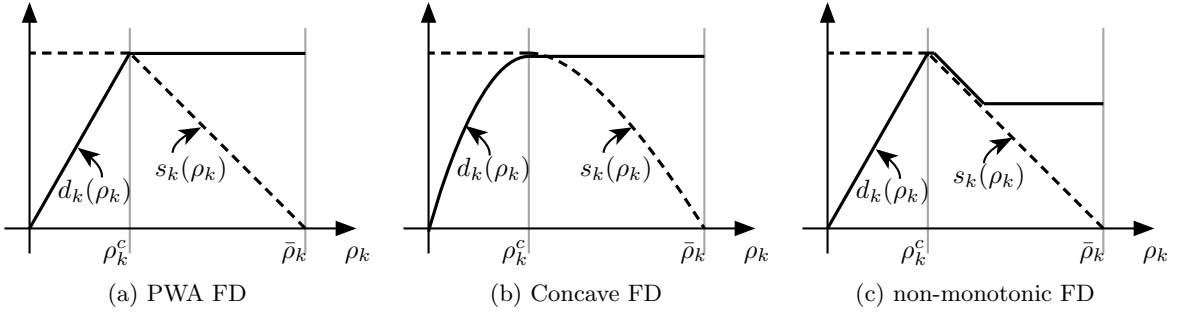
\begin{figure} 
	\centering
	\begin{subfigure}[b]{0.3\textwidth}
	  	\setlength{\unitlength}{0.1\textwidth}
  		\begin{picture}(10,9)
    			\put(0,0){\includegraphics[width=5cm]{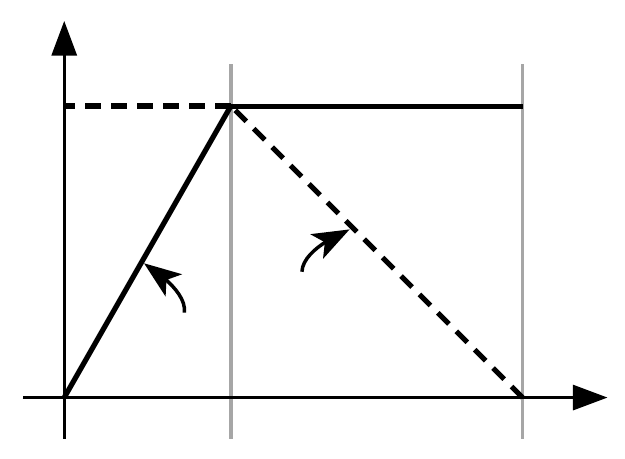}}
    			\put(9.5,0.2){$\rho_k$}
    			\put(3,0.2){$\rho^c_k$}
    			\put(7.8,0.2){$\bar\rho_k$}
    			\put(1.8,1.8){$d_k(\rho_k)$}
    			\put(4.4,2.4){$s_k(\rho_k)$}
 		\end{picture}
		\caption{PWA FD}
		\label{fig:ctm_fd_a}
	\end{subfigure} ~ 
	\begin{subfigure}[b]{0.3\textwidth}
	  	\setlength{\unitlength}{0.1\textwidth}
  		\begin{picture}(10,9)
    			\put(0,0){\includegraphics[width=5cm]{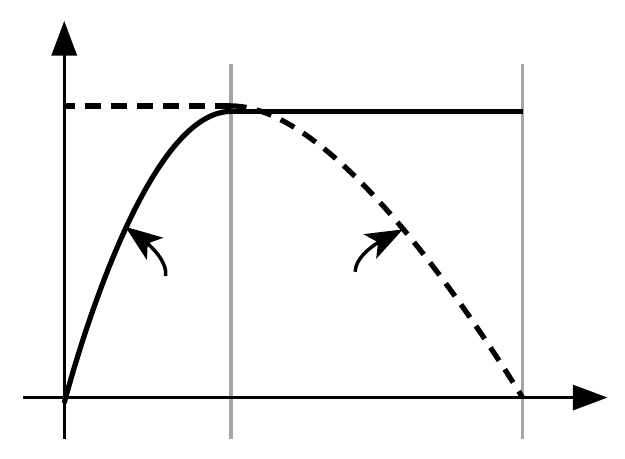}}
    			\put(9.5,0.2){$\rho_k$}
    			\put(3,0.2){$\rho^c_k$}
    			\put(7.8,0.2){$\bar\rho_k$}
    			\put(1.8,2.4){$d_k(\rho_k)$}
    			\put(5.1,2.4){$s_k(\rho_k)$}
 		\end{picture}
		\caption{Concave FD}	
		\label{fig:ctm_fd_b}
	\end{subfigure} ~ 
		\begin{subfigure}[b]{0.3\textwidth}
	  	\setlength{\unitlength}{0.1\textwidth}
  		\begin{picture}(10,9)
    			\put(0,0){\includegraphics[width=5cm]{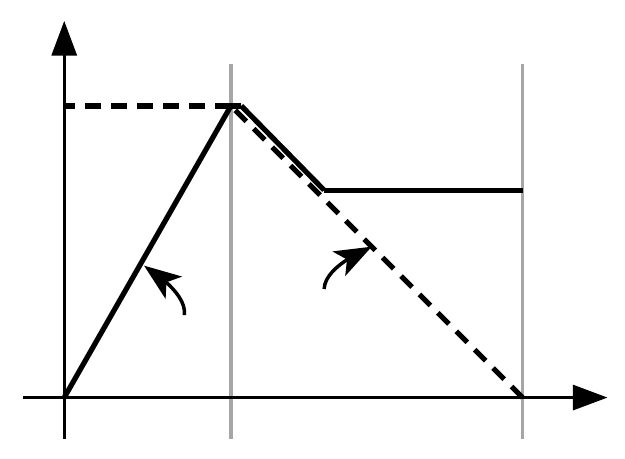}}
    			\put(9.5,0.2){$\rho_k$}
    			\put(3,0.2){$\rho^c_k$}
    			\put(7.8,0.2){$\bar\rho_k$}
    			\put(1.8,1.8){$d_k(\rho_k)$}
    			\put(4.4,2.2){$s_k(\rho_k)$}
 		\end{picture}
		\caption{non-monotonic FD}	
		\label{fig:ctm_fd_c}
	\end{subfigure} ~ 
	\caption{Different shapes of the fundamental diagram (FD) may be desirable in order to approximate real-world data. Figure (a) depicts the traditional PWA version. Figure (b) shows a version with monotonic demand and supply functions, satisfying Assumption \ref{assumption:bounded_slope} for suitable $\Delta t$. The depicted FD is also concave. Figure (c) shows a fundamental diagram with non-monotonic demand function, which does \emph{not} satisfy Assumption \ref{assumption:bounded_slope} for any $\Delta t$.}
	\label{fig:fd}
\end{figure}

We are now ready to formulate the main problem. Assume a freeway modeled by the monotonic CTM, subject to a demand profile $d_k(t)$ for $k \in \{0,\dots,n\}$ over a horizon $t \in \{0,\dots,T-1\}$. By controlling the metering rates, we seek to minimize the total time spent (TTS), given as the sum of the time drivers spent on the freeway and the waiting time on the metered onramps. The main problem is therefore given as:
\begin{subequations}
\begin{align}
\text{minimize} \quad & \TTS := \Delta t~\cdot \sum_{t=0}^{T} \sum_{k=1}^{n} \left( l_k \rho_k(t) + q_k(t) \right) \\
\text{subject to} \quad & \rho_k(t+1) = \rho_k(t) + \frac{\Delta t}{l_k} \left( \phi_{k-1}(t) + r_k(t) - \frac{1}{\bar\beta_k} \phi_k(t) \right) \\
 & q_k(t+1) = q_k(t) + \Delta t \cdot \left( w_k(t) - r_k(t) \right) \\
 & \phi_0(t) = w_0(t) \\
 & \phi_k(t) = \min\{d_k(\rho_k(t)), s_{k+1}(\rho_{k+1}(t))\} \label{eq:fd} \\ 
 & \phi_{n}(t) = d_n(\rho_n(t)) \\
 & 0 \leq r_k(t) \leq \bar r_k \label{eq:const_bounds}\\
 & 0 \leq q_k(t) \leq \bar q_k \label{eq:rampbounds2_alternative}\\
 & \rho_k(0), q_k(0), w_k(t) \text{ given} .
\end{align}
\label{eq:main_problem}
\end{subequations}
Note that the metering bounds \eqref{eq:rampbounds2} for time $t$ are equivalent to \eqref{eq:rampbounds2_alternative} at time $t+1$. Problem \eqref{eq:main_problem} is a standard problem in traffic control, which has been studied extensively. It is non-convex, because of the nonlinear fundamental diagram \eqref{eq:fd}. In \cite{gomes2006optimal}, it is shown that the special case of a piecewise-affine fundamental diagram admits an LP reformulation and hence can be solved efficiently. 
\begin{remark}
\label{remark:ramps}
In the statement of the main problem \eqref{eq:main_problem}, we introduce parameters $\bar r_k, \bar q_k$ and $\beta_k$ for every cell. In practice, not every cell will be equipped with both an onramp and an offramp and many cells might have neither. The equations \eqref{eq:main_problem} are general enough to capture all these situations and one can simply select $\bar r_k = \bar q_k = w_k(t) = 0 ,~\forall t,$ if no onramp is present in cell $k$. Similarly, choosing $\beta_k = 0$, which in turn implies that $\bar\beta_k = 1$, means no offramp is present in cell $k$.
\end{remark}

\section{A distributed controller} 
\label{sec:local_feedback}

The problem of minimizing TTS represents the global perspective, in which we optimize the whole system over the complete horizon. By contrast, one might also seek to optimize the here-and-now performance. This idea can be formalized by introducing the Total Distance Traveled (TDT), defined as
\begin{align*}
\text{TDT}(t) := \Delta t \cdot \sum_{k=1}^{n} l_k \phi_k(t) ,
\end{align*}
which is simply the total distance traveled by all vehicles on the mainline within time interval $t$. In this section, we formulate an one-step look-ahead controller, which maximizes TDT for the next time step in a greedy fashion:
\begin{subequations}
\begin{align}
\text{maximize} \quad & \text{TDT}(t+1) = \Delta t \cdot \sum_{k=1}^{n} l_k \phi_k(t+1)  \\
\text{subject to} \quad & \rho_k(t+1) = \rho_k(t) + \frac{\Delta t}{l_k} \left( \phi_{k-1}(t) + r_k(t) - \frac{1}{\bar\beta_k} \phi_k(t) \right) \label{eq:rho} \\
 & q_k(t+1) = q_k(t) + \Delta t \cdot \left( w_k(t) - r_k(t) \right) \\
 & \phi_0(t+1) = w_0(t+1) \\
 & \phi_k(t+1) = \min\{d_k(\rho_k(t+1)), s_{k+1}(\rho_{k+1}(t+1))\} \\ 
 & \phi_{n}(t+1) = d_n(\rho_n(t+1)) \\
 & 0 \leq r_k(t) \leq \bar r_k \label{eq:one_step_ahead_rampbounds1} \\
 & 0 \leq q_k(t+1) \leq \bar q_k \label{eq:one_step_ahead_rampbounds2} \\
 & \rho_k(t),q_k(t),w_k(t) \text{ given} . \label{eq:step_k}
\end{align}
\label{eq:one_step_ahead}
\end{subequations}
It is straightforward to show that the following distributed feedback controller provides a (non-unique) explicit solution to this  optimization problem. The explicit, feedback solution requires perfect model knowledge, but it does not rely on any online-optimization or traffic demand prediction:

\begin{lemma} 
An explicit solution to the one-step-ahead optimal control problem \eqref{eq:one_step_ahead} is given by the following feedback policy:
\begin{align}
r_k^*(t) := \left[ \frac{l_k}{\Delta t} \left( \rho^c_k - \rho_k(t) \right) + \frac{\phi_k(t)}{\bar\beta_k} - \phi_{k-1}(t) \right]_{\max \left\{ 0, \frac{1}{\Delta t} \left( q_k(t) - \bar q_k \right) + w_k(t) \right\} }^{\min   \left\{ \bar r_k, \frac{1}{\Delta t} q_k(t) + w_k(t) \right\}} \quad \forall k \in \{1,\dots,n\} .
\label{eq:best_effort}
\end{align}
\label{lemma:best_effort}
\end{lemma}
\begin{proof} 
Consider an arbitrary flow $\phi_k(t+1)$, $k \in \{1,\dots,n-1\}$. It is easy to see that it only depends on the ramp metering rates in the adjacent cells $k$ and $k+1$ at time $t$, since we consider a one-step-ahead control problem. More precisely, it depends on the traffic demand $d_k(\rho_k(t+1))$ and the supply of free space $s_{k+1}(\rho_{k+1}(t+1))$. Consider the problem of maximizing this flow 
\begin{align*}
\max_{r(t)} ~ & \phi_k(t+1) 
=
\max_{r(t)} ~ \min\{d_k (\rho_k(t+1)), s_{k+1}(\rho_{k+1}(t+1))\} \\
&\leq
\min \big\{
\underbrace{\max_{r(t)} ~ d_k (\rho_k(t+1))}_{(1)}, ~
\underbrace{\max_{r(t)} ~ s_{k+1}(\rho_{k+1}(t+1))}_{(2)}
\big\}
\end{align*}
alone. It turns out that the policy \eqref{eq:best_effort} maximizes both terms in the previous equation. In order to show this, we resort to a case distinction:
\begin{itemize}
\item[(i)] If $r_k^*(t)$ does not saturate at the upper bound, it follows that $r_k^*(t) \geq \frac{l_k}{\Delta t} \left( \rho^c_k - \rho_k(t) \right) + \frac{\phi_k(t)}{\bar\beta_k} - \phi_{k-1}(t)$ and hence $\rho_k(t+1) = \rho_k(t) + \frac{\Delta t}{l_k} \left( \phi_{k-1}(t) + r^*_k(t) - \frac{1}{\bar\beta_k} \phi_k(t) \right) \geq \rho_k^c$. Therefore, $d_k(\rho_k(t+1)) = d_k(\rho_k^c)$ and the first term is maximized.
\item[(ii)] Otherwise, recall that $d_k(\cdot)$ is nondecreasing. Since $r_k^*(t)$ saturates at the upper bound w.r.t\ \eqref{eq:one_step_ahead_rampbounds1} and \eqref{eq:one_step_ahead_rampbounds2} of the feasible range, the first term is maximized.
\end{itemize}
Similarly, we can analyze the effects of the metering rate at cell $k+1$:
\begin{itemize}
\item[(i)] If $r_{k+1}^*(t)$ does not saturate at the lower bound, it follows that $r_{k+1}^*(t) \leq \frac{l_{k+1}}{\Delta t} \left( \rho^c_{k+1} - \rho_{k+1}(t) \right) + \frac{\phi_{k+1}(t)}{\bar\beta_{k+1}} - \phi_{k}(t)$ and hence $\rho_{k+1}(t+1) = \rho_{k+1}(t) + \frac{\Delta t}{l_{k+1}} \left( \phi_{k}(t) + r^*_{k+1}(t) - \frac{1}{\bar\beta_{k+1}} \phi_{k+1}(t) \right) \leq \rho_{k+1}^c$. Therefore, $s_{k+1}(\rho_{k+1}(t+1)) = s_{k+1}(\rho_{k+1}^c)$ and the second term is maximized.
\item[(ii)] Otherwise, recall that $s_{k+1}(\cdot)$ is nonincreasing. Since $r_{k+1}^*(t)$ does saturate at the lower bound of the feasible range, the second term is maximized.
\end{itemize}

Since both the first and the second term are maximized, we conclude that the proposed feedback law maximizes $\phi_k(t+1)$, for all $k \in \{1,\dots,n-1\}$. The case of $\phi_n(t+1)$ is similar, except that the second term is not present. Also $\phi_0(t+1)$ does not depend on the metering rates. All flows are jointly maximized and therefore, TDT is maximized as well.
\end{proof}

This controller aims at moving the local density to the critical density as fast as possible and in that sense, it is very similar in spirit to many existing and practically successful ramp metering policies, most notably Alinea \cite{papageorgiou1991alinea}. In the following, we will be referring to this one-step-ahead controller as the \emph{best-effort} (BE) controller. We stress that we do \emph{not} propose this controller as a new approach for ramp metering, as it relies on perfect knowledge of the fundamental diagram and the split ratios, which is unrealistic in practice. Rather, the best-effort controller will serve both as a useful tool in computing (bounds on) the optimal solution of the problem of minimizing TTS and as a proxy for the practical Alinea controller, for which a direct theoretical analysis would be difficult. 

There has been speculation in \cite{zhang2004optimal} on whether a control policy similar to the feedback law \eqref{eq:best_effort} minimizes TTS over the complete horizon. We will present two counterexamples showing that this is not necessarily the case in all scenarios.

\begin{example}[Lower bound saturation] 
\label{ex:lb}
Consider a two-cell freeway as depicted in Figure \ref{fig:counterexample_lb}, with a metered onramp at the second cell and an offramp with a large split ratio at the first cell, in this example $\beta_1 = 0.8$. Some traffic demand is present at the onramp in the beginning, but a spike in mainline demand is arriving on the mainline shortly after. The situation is sketched in Figure \ref{fig:counterexample_lb1}. Because of the large split ratio at the offramp, a spill back of congestion into the first cell severely reduces the total discharge flows and hence increases TTS. Once the spike in demand arrives, the constraint $0 \leq r_k(t)$ becomes active. However, cars already admitted to the freeway from the onramp can not be retracted to the onramp queue. For suitably chosen parameters, the optimal policy is to preemptively reduce the density in the second cell below the critical density, as depicted in Figure \ref{fig:counterexample_lb2}. The resulting reduction in outflow from the mainline is more than compensated by the increased discharge from the offramp. The demand is piecewise constant: $d_0(t) = 5000$ for $2\text{min} \leq t \leq 4 \text{min}$, $d_0(t) = 0$ otherwise and $d_2(t) = 2500$ for $0 \text{min} \leq t \leq 2 \text{min}$, $d_2(t) = 0$ otherwise. The parameters are chosen as $l_1 = l_2 = 1$km, $v_1 = v_2 = w_1 = w_2 = 100$km/h, $\rho_1^c = 100$(cars)/km, $\rho_2^c = 10$(cars)/km, $\bar\rho_1 = 200$(cars)/km, $\bar\rho_2 = 20$(cars)/km, $F_1 = 1000$(cars) and $F_2 = 500$(cars). Note that the parameter values have been chosen to amplify the suboptimality of the policy for illustration purposes, without considering whether such values are realistic. 
\begin{figure}[t] 
	\centering
	\begin{subfigure}[b]{0.38\textwidth}
		\includegraphics[width=\textwidth]{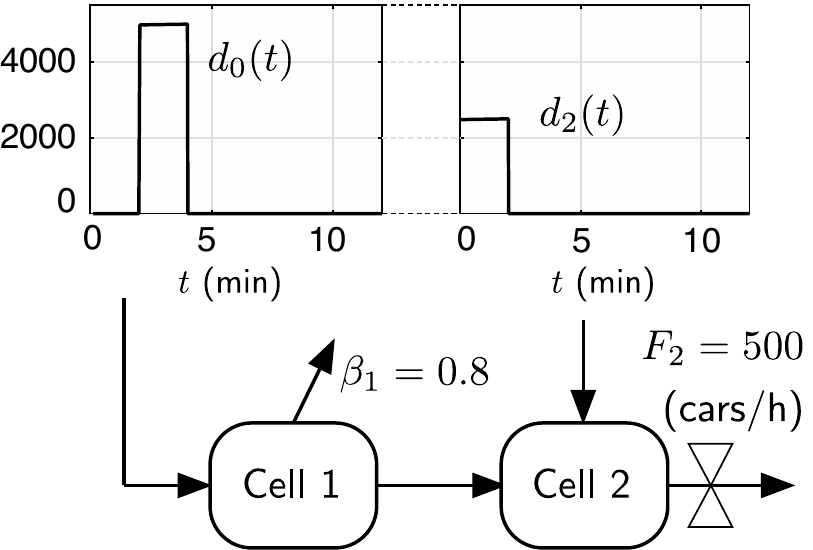}
		\caption{Freeway topology and demand patterns}
		\label{fig:counterexample_lb1}
	\end{subfigure} ~	
	\begin{subfigure}[b]{0.29\textwidth}
		\includegraphics[width=\textwidth]{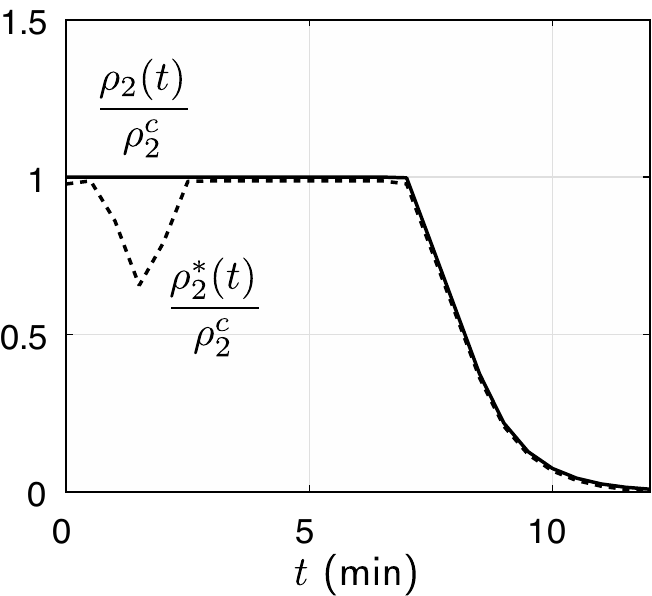}
		\caption{Density in cell 2}
		\label{fig:counterexample_lb2}
	\end{subfigure} ~
	\begin{subfigure}[b]{0.29\textwidth}
		\includegraphics[width=\textwidth]{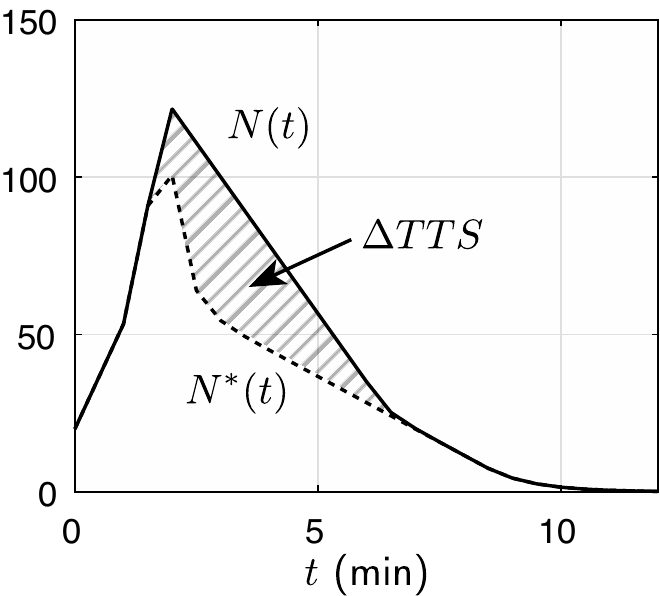}
		\caption{Total number of cars}
		\label{fig:counterexample_lb3}
	\end{subfigure}
	\caption{In case of a spike in mainline demand, it might be advisable to preemptively reduce the local density at a bottleneck below the critical density, in oder to accomodate (parts of) the traffic wave in the free space and avoid spill back of congestion. In Figure \ref{fig:counterexample_lb3}, the variables $N(t)$ and $N^*(t)$ represent the total number of cars on the freeway and at the onramps, so $N(t) := l_1 \rho_1(t) + l_2 \rho_2(t) + q_2(t)$ and likewise for $N^*(t)$. Therefore, the savings, i.e.,\ the difference in TTS between the two controllers, are proportional to the highlighted region.}
	\label{fig:counterexample_lb}
\end{figure} 
\end{example} 

Conversely, we can also construct an example in which the upper bound on the ramp metering rate prevents optimality of the BE controller:

\begin{example}[Upper bound saturation] 
\label{ex:ub}
Consider a one cell freeway with a metered onramp, as sketched in Figure \ref{fig:counterexample_ub1}. Some traffic demand is present at the onramp. A spike in mainline demand arrives at the beginning of the considered time interval, but afterwards, the mainline demand decays to zero. In case of BE control, the initial spike in mainline demand causes a congestion and hence BE control will use ramp metering to hold cars back on the queue. But when the mainline density decays again, ramp metering will not be able to release cars sufficiently fast to keep the density at the critical density due to inherent limits on the flow from the onramp, which require $r_k(t) \leq \bar r_k$. The optimal policy does not use ramp metering at all and allows for a larger congestion in the beginning, as depicted in Figure \ref{fig:counterexample_ub2}. Since there is no offramp, spill back of congestion is not an issue in this scenario. The demand is piecewise constant: $d_0(t) = 4000$ for $0\text{min} \leq t \leq 3 \text{min}$, $d_0(t) = 0$ otherwise and $d_1(t) = 1800$ for $0\text{min} \leq t \leq 5 \text{min}$, $d_1(t) = 0$ otherwise. The parameter values are chosen as $l_1 = 1$km, $v_1 = 100$km/h, $w_1 = 25$km/h, $\rho_1^c = 50$(cars)/km, $\bar\rho_1 = 250$(cars)/km and $F_1 = 5000$(cars).  Again, the parameter values have been chosen such as to amplify the suboptimality of the policy for illustration purposes. 
\begin{figure}[tp]
	\centering
		\begin{subfigure}[b]{0.38\textwidth}
		\includegraphics[width=\textwidth]{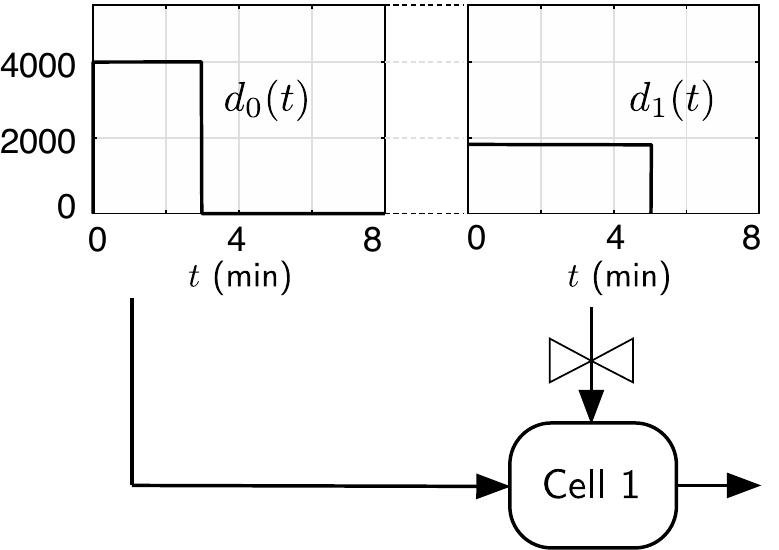}
		\caption{Freeway topology and demand patterns}
		\label{fig:counterexample_ub1}
	\end{subfigure} ~	
	\begin{subfigure}[b]{0.29\textwidth}
		\includegraphics[width=\textwidth]{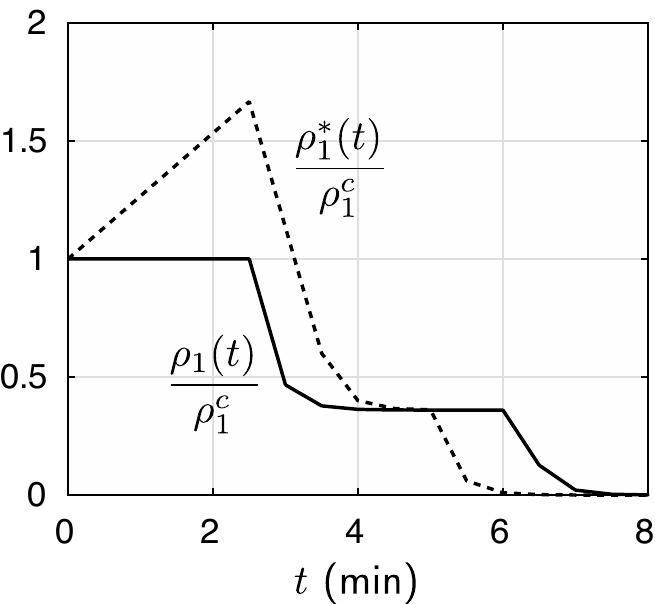}
		\caption{Density in cell 2}
		\label{fig:counterexample_ub2}
	\end{subfigure} ~
	\begin{subfigure}[b]{0.29\textwidth}
		\includegraphics[width=\textwidth]{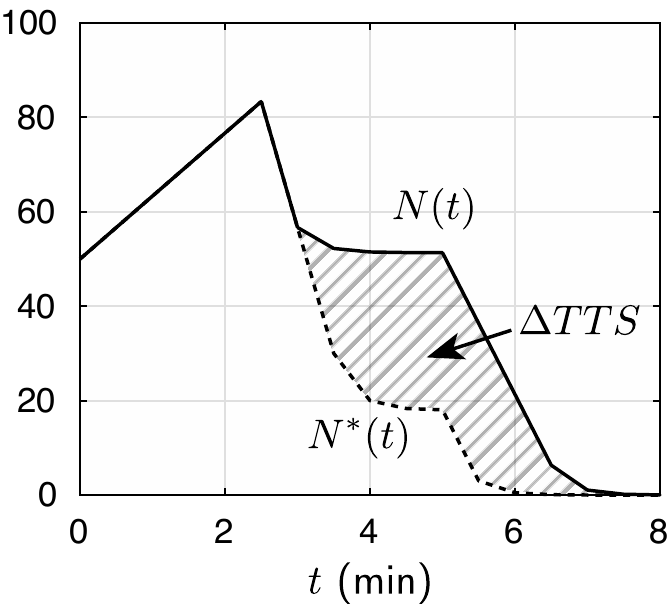}
		\caption{Total number of cars}
		\label{fig:counterexample_ub3}
	\end{subfigure}	
	\caption{If no spill back of the congestion is possible, it might be preferable (in the monotonic CTM) to allow a congestion to form on the mainline in order to ease the burden on the onramp, which becomes a bottleneck later on. The variables $N(t)$ and $N^*(t)$ represent the total number of cars on the freeway and at the onramps, so $N(t) := l_1 \rho_1(t) + q_1(t)$ and likewise for $N^*(t)$. Therefore, the savings, i.e.,\ the difference in TTS between the two controllers, are proportional to the highlighted region.}
	\label{fig:counterexample_ub}
\end{figure}
\end{example}

Closer inspection reveals that whether the ramp metering bounds \eqref{eq:rampbounds1} are active or not is essential for deriving a certificate of optimality. Indeed, in Section \ref{sec:optimality_conditions},  we derive explicit conditions relating the activation of these constraints to the optimality of the best-effort controller.

\section{Sufficient optimality conditions}  
\label{sec:optimality_conditions}

It turns out that even in the presence of controller saturation, sufficient optimality conditions can be found that ensure minimization of TTS. To derive these, we first apply a linear state transformation to the CTM to obtain a system description better suitable for our analysis. Then, we introduce conditions that characterize the state of onramps at individual time steps and show how these conditions, together with monotonicity of the system in the new coordinates, ensure global minimization of TTS.

Let us introduce the \emph{cumulative flow} $\Phi_k(t)$, defined as
\begin{align*}
 \Phi_k(t) =: \Phi_k(0) + \Delta t \cdot \sum_{\tau = 0}^{t-1} \phi_k(\tau) .
\end{align*}
with $\Phi_k(0) := \sum_{j = k+1}^{n} \frac{1}{\bar\beta_{(k+1,j-1)}} \left( l_j \rho_j(0) \right)$ and $\bar\beta_{(k,j)} := \prod_{i = k}^{j} \bar\beta_i$ if $k \leq j$ and $\bar \beta_{(k+1,k)} := 1$. The constant offset accounts for the initial density of the freeway, that is cars that have entered the freeway and travelled to cell $k$ or further before the beginning of the considered time horizon. This offset is convenient when expressing densities in terms of cumulative flow; it does not affect our optimality arguments.. Note that $\Phi_k(t)$ is dimensionless and can be interpreted as the number of cars that have passed cell $k$. Such a quantity, defined as a flow aggregate over time, is reminiscent of quantities used in network calculus, like \emph{cumulative arrivals} and \emph{cumulative departures}. Network calculus has been applied to the control of traffic networks \cite{varaiya2013max}, however, it seems that so far only store-and-forward models without congestion spill-back have been considered, but not the CTM.

We can we express the cumulative flow directly in terms of the original system states: first, we extend the model to include an additional cell with index $n+1$ at the end of the freeway. This cell is defined to have infinite storage capacity and collects all of the flow that leaves the freeway via the mainline. An infinite storage capacity can be formalized by choosing $s_{n+1}(\rho_{n+1}) \equiv +\infty$ such that all of the upstream demand is admitted and $d_{n+1}(\rho_{n+1}) \equiv 0$ such that no car leaves the cell. The purpose of this cell is to assist in bookkeeping of all the cars that have travelled through the freeway and it will therefore \textit{not} be considered in the computation of the Total Time Spent. Note that some cars also leave the freeway via the offramps, but because of the assumption of constant split ratios, we can reconstruct the outflows from the densities in cells $k \in \{1,\dots,n+1\}$. Let us also define the total inflow up to and including time $t$ as $R_k(t) := \Delta t \cdot \sum_{\tau = 0}^{t-1} r_k(\tau)$ and the cumulative, external demands as $W_k(t) := q_k(0) + \Delta t \cdot \sum_{\tau = 0}^{t-1} w_k(\tau)$. Both quantities are dimensionless and can be interpreted as the number of cars that have entered the freeway or arrived at the onramp, respectively\footnote{Note however, that the CTM is an averaged model, so neither $R_k(t)$ nor $W_k(t)$ are restricted to integer values.}.

We can now express the cumulative flows 
\begin{align}
\Phi_k(t) = \sum_{j = k+1}^{n+1} \frac{1}{\bar\beta_{(k+1 , j-1)}} \left( l_j \rho_j(t) - R_j(t) \right) .
\label{eq:actm_density}
\end{align}
as linear functions of the original states and inputs of the CTM. The whole CTM can be expressed equivalently in terms of the cumulative quantities:
\begin{align}
\label{eq:aggregated_ctm}
\left. \begin{array}{rrl}
\Phi_k(t+1) &= f_k (\Phi(t),R(t)) &:= \Phi_k(t) + \Delta t~ \phi_k(t) , \\ [1.5ex]
\text{with: } & \phi_0(t) &= w_0(t), \\
 & \phi_k(t+1) &= \min\{d_k(\rho_k(t+1)), s_{k+1}(\rho_{k+1}(t+1))\} , \\  
 & \phi_{n}(t) &= d_n(\rho_n(t))  \\ [1ex]
\text{and: } & \rho_k(t) &= \frac{1}{l_k} \left( \Phi_{k-1}(t) - \frac{1}{\bar\beta_k} \Phi_{k}(t) + R_k(t) \right) ,
\end{array} \right\} \forall k \in \{1,\dots,n\}
\end{align}
in which the cumulative flows $\Phi_k(t)$ are the states and the inflows $R_k(t)$ are the inputs. Note that we have introduced the shorthand notation $f_k(\Phi(t),R(t))$ for the systems equations. We will also use
$f(\cdot,\cdot) := \begin{bmatrix} f_1(\cdot,\cdot) \dots f_n(\cdot,\cdot) \end{bmatrix}^{\top}$.
The input constraints \eqref{eq:rampbounds1} and \eqref{eq:rampbounds2} can be expressed in term of the cumulative quantities as
\begin{subequations}
\label{eq:aggregated_rampbounds}
\begin{align}
R_k(t-1) &\leq R_k(t) \leq R_k(t-1) + \bar r_k , \\
W_k(t) - \bar q_k &\leq R_k(t) \leq W_k(t) . \label{eq:agg_rampbound2}
\end{align}
\end{subequations}
We will refer to the system described by equations \eqref{eq:aggregated_ctm} and \eqref{eq:aggregated_rampbounds} as the \emph{cumulative CTM} (CCTM). Monotonicity properties will facilitate the analysis of optimal control problems for the CCTM. 

\begin{lemma} 
\label{prop:monotonic_in_flows}
The systems equations $f_k(\Phi(t),R(t))$ for all $0 \leq k \leq n$ in the CCTM are nondecreasing in the cumulative flows $\Phi(t)$.
\end{lemma}
\begin{proof} 
Using the definition of the CCTM \eqref{eq:aggregated_ctm}, we can write equivalently
\begin{equation*} 
\Phi_k(t+1) = \min \left\{ \Phi_k(t) + \Delta t \cdot d_k(\rho_k(t)), \Phi_k(t) + \Delta t \cdot s_{k+1}(\rho_{k+1}(t)) \right\}
\end{equation*}
for $k \in \{1,\dots,n-1\}$. The minimum of monotonic functions is monotonic. We can therefore verify monotonicity of the CTM by checking that both of the functions
\begin{align*}
f^{(-)}_k(\Phi_{k-1}(t), \Phi_k(t)) &:=  \Phi_k(t) + \Delta t \cdot d_k \left( \frac{1}{l_k} \left( \Phi_{k-1}(t) - \frac{1}{\bar\beta_k} \Phi_{k}(t) + R_k(t) \right) \right) , \\
f^{(+)}_k(\Phi_k(t), \Phi_{k+1}(t)) &:= \Phi_k(t) + \Delta t \cdot s_{k+1} \left( \frac{1}{l_{k+1}} \left( \Phi_{k}(t) - \frac{1}{\bar\beta_{k+1}} \Phi_{k+1}(t) + R_{k+1}(t) \right) \right) ,
\end{align*}
are nondecreasing in $\Phi_{k-1}(t), \Phi_k(t)$ and $\Phi_{k+1}(t)$.
\begin{itemize} 
\item[(i)] To verify monotonicity in $\Phi_{k-1}(t)$, first note that $f^{(+)}_k$ does not depend on $\Phi_{k-1}(t)$, so it is trivially nondecreasing. 
Furthermore, $\rho_k(t)$ is nondecreasing in $\Phi_{k-1}(t)$, since by assumption $l_k >0$. Also, $d_k(\cdot)$ is nondecreasing according to Assumption \ref{assumption:bounded_slope}. Recalling that $\Delta t > 0$, we conclude that $f^{(-)}_k$, which is a composition of the previously analyzed functions, is nondecreasing in $\Phi_{k-1}(t)$.
\item[(ii)] To verify monotonicity in $\Phi_{k+1}(t)$, first note that $f^{(-)}_k$ does not depend on $\Phi_{k+1}(t)$, so it is trivially nondecreasing. 
Furthermore, $\rho_k(t)$ is nondecreasing in $\Phi_{k+1}(t)$, since by assumption $l_k >0$ and $\bar\beta_{k+1} > 0$. Also, $s_{k+1}(\cdot)$ is nonincreasing according to Assumption \ref{assumption:bounded_slope}. Recalling that $\Delta t > 0$, we conclude that $f^{(+)}_k$, which is a composition of the previously analyzed functions, is nondecreasing in $\Phi_{k+1}(t)$.
\item[(iii)] To verify monotonicity of $f^{(-)}_k$ in $\Phi_k(t)$, recall that according to Assumption \ref{assumption:bounded_slope}, $d_k(\cdot)$ is Lipschitz continuous, so $ | d_k(x+\Delta x) - d_k(x) | \leq c_k^d | \Delta x | $. Therefore:
\begin{align*}
f^{(-)}_k&(\Phi_{k-1}(t), \Phi_k(t) + \Delta \Phi) - f^{(-)}_k(\Phi_{k-1}(t), \Phi_k(t)) \\
&= \Delta \Phi +  \Delta t \cdot \left[ d_k \left( \frac{1}{l_k} \left( \Phi_{k-1}(t) - \frac{1}{\bar\beta_k} ( \Phi_{k}(t) + \Delta \Phi ) + R_k(t) \right) \right) \right. \\
& \left. \hspace{1cm} - d_k \left( \frac{1}{l_k} \left( \Phi_{k-1}(t) - \frac{1}{\bar\beta_k} \Phi_{k}(t) + R_k(t) \right) \right)  \right] \\
& \geq  \Delta \Phi - \Delta t \cdot c_k^d \cdot \left| \frac{1}{l_k \bar\beta_k} \Delta \Phi \right| \\
& \geq \Delta \Phi \cdot \left( 1 - \Delta t \cdot \frac{l_k \bar\beta_k}{\Delta t} \cdot \frac{1}{l_k \bar\beta_k} \right) = 0 .
\end{align*}
In the first inequality, we use Lipschitz continuity and in the last inequality, we replace the Lipschitz constant for the demand function with its upper bound from Assumption \ref{assumption:bounded_slope}. Monotonicity of $f^{(+)}_k$ in $\Phi_k(t)$ can be shown in a similar way.
\end{itemize} 
 Also, the cases $k=0$ and $k=n$ follow along the same lines.
\end{proof}

Similarly, the dependency of the states on the inputs is characterized by the following relationship:
\begin{lemma} 
\label{prop:monotonic_in_inputs}
The systems equation $f_k(\Phi(t),R(t))$ for $k$ fixed ($1 \leq k \leq n-1$)  is nondecreasing in input $R_k(t)$ and nonincreasing in input $R_{k+1}(t)$. Also, $f_0(\Phi(t),R(t))$ is nonincreasing in $R_1(t)$ and $f_n(\Phi(t),R(t))$ is nondecreasing in $R_n(t)$.
\end{lemma} 
The proof of Lemma \ref{prop:monotonic_in_inputs} follows the same ideas as the proof of Lemma \ref{prop:monotonic_in_flows} and can be found in Appendix \ref{proof:monotonic_in_flows}. Next, we will define conditions that will be useful in deriving sufficient optimality conditions for best-effort control.

\begin{definition} 
\label{def:local_optimality}
A cell $k$ with metered onramp is called \emph{restrictive} at time $t$ if
\begin{itemize}
\item[(i)] the onramp is not completely full $q_k(t) < \bar q_k$ and the upstream flow into the cell is limited by non-maximal supply of free space, i.e.,\ $s_k(\rho_{k}(t)) = \phi_{k-1}(t) < F_{k-1}$, or 
\item[(ii)] the onramp is not empty $ q_k(t) > 0$ and the downstream flow out of the cell is limited by non-maximal demand, i.e.,\  $d_k(\rho_k(t)) = \phi_k(t) < F_k$. 
\end{itemize}
Otherwise, we call the cell \emph{nonrestrictive}.
\end{definition} 

This definition formalizes the intuitive idea to operate each onramp so as to maximize local mainline flow in the subsequent time step $t+1$: In case (i), the existence of free space on the onramp indicates that a potential congestion on the mainline, which is equivalent to the flow being restricted by the traffic supply, might have been prevented or at least reduced by keeping additional cars on the onramp. Conversely, the existence of free-flow conditions on the mainline below the maximal flow, as indicated by the restriction of the flow by the traffic demand, raises the question as to why any cars should be held back on the onramp queue. The idea will be illustrated in the following example.

\begin{figure} 
	\centering
	\begin{subfigure}[b]{0.2\textwidth}
	  	\setlength{\unitlength}{0.1\textwidth}
  		\begin{picture}(10,10)
    			\put(-0.1,1){\includegraphics[width=3.5cm]{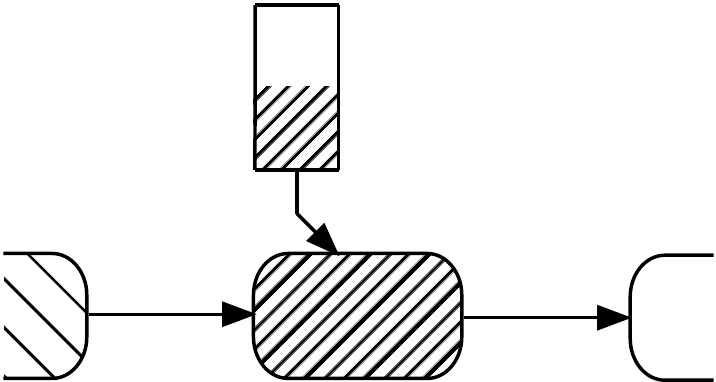}}
    			\put(-0.6,4.1){$\begin{array}{l} \rho_{k-1} \\> 0 \end{array}$}
    			\put(10.2,1.5){\Large ?}
    			\put(3.9,0){$\rho_{k} = \bar\rho_{k}$}
    			\put(2.7,7.4){$0 < q_{k} < \bar q_k$}
 		\end{picture}
		\caption{Restrictive}
		\label{fig:restrictive_ex1}
	\end{subfigure} \hspace{0.7cm} 
	\begin{subfigure}[b]{0.2\textwidth}
	  	\setlength{\unitlength}{0.1\textwidth}
  		\begin{picture}(10,10)
    			\put(-0.1,1){\includegraphics[width=3.5cm]{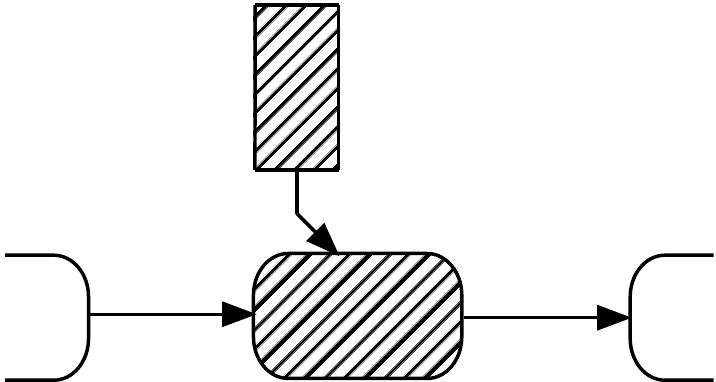}}
    			\put(-0.1,1.5){\Large ?}
    			\put(10,1.5){\Large ?}
    			\put(3.9,0){$\rho_{k}^c < \rho_k$}
    			\put(2.7,7.4){$q_{k} = \bar q_k$}
 		\end{picture}
		\caption{Nonrestrictive}
		\label{fig:restrictive_ex2}
	\end{subfigure} \hspace{0.7cm} 
		\begin{subfigure}[b]{0.2\textwidth}
	  	\setlength{\unitlength}{0.1\textwidth}
  		\begin{picture}(10,10)
    			\put(-0.1,1){\includegraphics[width=3.5cm]{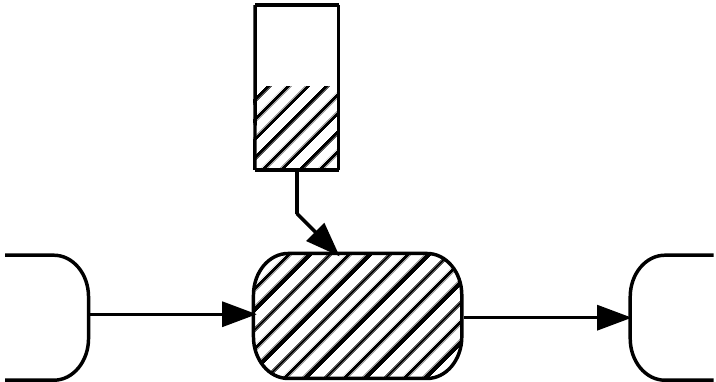}}
    			\put(-0.6,4.1){$\begin{array}{l} \rho_{k-1} \\= 0 \end{array}$}
    			\put(10.1,1.5){\Large ?}
    			\put(2.4,0){$\rho_k^c < \rho_{k} < \bar\rho_k$}
    			\put(2.7,7.4){$0 < q_{k} < \bar q_k$}
 		\end{picture}
		\caption{Nonrestrictive}
		\label{fig:restrictive_ex3}
	\end{subfigure} \hspace{0.7cm} 
		\begin{subfigure}[b]{0.2\textwidth}
	  	\setlength{\unitlength}{0.1\textwidth}
  		\begin{picture}(10,10)
    			\put(0,1){\includegraphics[width=3.5cm]{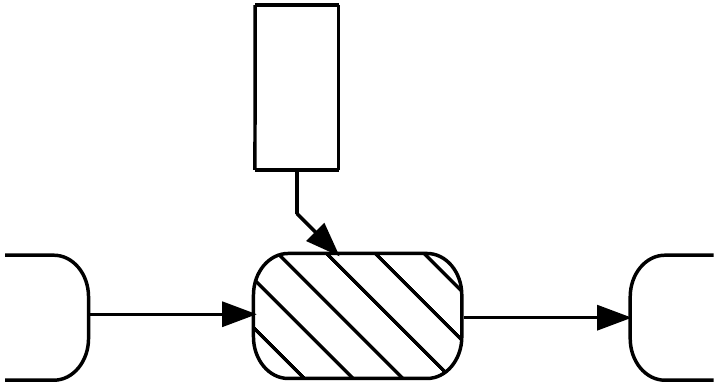}}
    			\put(-0.1,1.5){\Large ?}
    			\put(10,1.5){\Large ?}
    			\put(3.9,0){$\rho_{k} = \rho_{k}^c$}
    			\put(4.1,5){\Large ?}
 		\end{picture}
		\caption{Nonrestrictive}
		\label{fig:restrictive_ex4}
	\end{subfigure} ~ 
	\caption{Example situations in which cell $k$ is restrictive or nonrestrictive. A question mark indicates that the corresponding density or ramp occupancy is not relevant for determining if cell $k$ is restrictive or not.}
	\label{fig:restrictive_ex}
\end{figure}
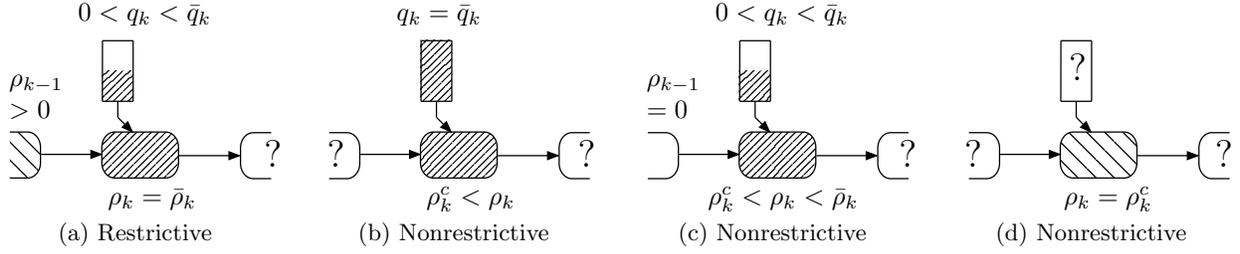

\begin{example}
\label{ex:restrictive}
Consider the situations sketched in Figure \ref{fig:restrictive_ex}, which exemplify when a cell is restrictive. In all these examples, we assume that $s_k(\rho_k(t)) > 0$ if $\rho_k(t) < \bar\rho_k$ and $d_k(\rho_k(t)) > 0$ if $\rho_k(t) > 0$. 
\begin{itemize}
\item[(i)] A completely congested cell implies $s_k(\rho_k(t)) = s_k(\bar\rho_k) = 0$. Since the upstream cell is nonempty, it follows that $s_k(\rho_k(t)) = 0 < d_{k-1}(\rho_{k-1}(t))$. But even though cell $k$ is completely congested, the onramp is not completely filled $q_k(t) < \bar q_k$. Therefore, cell $k$ is restrictive.
\item[(ii)] Cell $k$ is congested $\rho_k(t) > \rho_k^c$ , but in contrast to case (i), the maximal number of cars is stored on the onramp $q_k(t) = \bar q_k$. Therefore, this cell is nonrestrictive.
\item[(iii)] Cell $k$ is congested $\rho_k^c < \rho_k(t) < \bar\rho_k$ and the onramp is not completely filled $q_k(t) < \bar q_k$. However, in contrast to case (i), the upstream cell $k+1$ is empty, therefore, the congestion in cell $k$ does not obstruct the mainline flow $s_k(\rho_k(t)) > 0 = d_{k-1}(0)$ and cell $k$ is nonrestrictive
\item[(iv)] Cell $k$ is operating exactly at the critical density. Therefore, $s_k(\rho_k(t)) = s_k(\rho_k^c) \geq F_{k-1}$ and $d_k(\rho_k(t)) = d_k(\rho_k^c) \geq F_k$ and the cell is nonrestrictive, for arbitrary upstream demand, downstream supply and local ramp occupancy.
\end{itemize}
\end{example}

\begin{remark}
Definition \ref{def:local_optimality} characterizes cells with onramps. Proofs in the remainder of this section will analyze pairs of adjacent cells, to determine the behavior of the flow between these cells. To avoid unnecessary case distinctions with regard to which of these cells are equipped with onramps and which are not, note that cells without an onramp can still be expressed in the CCTM framework by $\bar r_k = \bar q_k = W_k(t) = 0 , ~\forall t,$ for all cells $k$ without an onramp, exactly in the same way as described in Remark \ref{remark:ramps} for the CTM. Then constraints \eqref{eq:aggregated_rampbounds} imply $R_k(t) = 0 , ~\forall t$, as desired. For such cells $k$, it always holds that $0 = q_k(t) = \bar q_k(t)$ and therefore, these cells are classified as \emph{nonrestrictive} at all times.
\end{remark}

It is important to note that a cell operating exactly at the critical density $\rho_c$ is always nonrestrictive in the sense of Definition \ref{def:local_optimality}: According to Lemma \ref{lemma:best_effort}, both upstream and downstream flow are jointly maximized at this density. In fact, the best-effort controller, which tracks the critical density locally, is able to keep cells nonrestrictive under certain conditions, as explained by the following lemma:
\begin{lemma} 
\label{lemma:feedback_optimality}
Assume that for a freeway controlled by the best-effort controller, the onramp at cell $k$ at time $t$ satisfies
\begin{align*}
0 \leq \left[ \frac{l_k}{\Delta t} \left( \rho^c_k - \rho_k(t) \right) + \frac{\phi_k(t)}{\bar\beta_k} - \phi_{k-1}(t) \right]^{ \frac{1}{\Delta t} q_k(t) + w_k(t) }_{ \frac{1}{\Delta t} \left( q_k(t) - \bar q_k \right) + w_k(t) } \leq \bar r_k,
 \end{align*}
 i.e., the best-effort policy \eqref{eq:best_effort} is not affected by the constraints \eqref{eq:rampbounds1}. Then, cell $k$ is nonrestrictive at time $t+1$.
\end{lemma} 
\begin{proof}[Sketch of proof] 
Note that this condition is sufficient, but not necessary. Sufficiency can be readily verified by plugging the feedback policy in the CTM equations and making a case distinction for each affine piece of the control law.
\end{proof} 
We will now use monotonicity properties of the CCTM to show that the cumulative flows are also maximized over longer periods of time, if the freeway is nonrestrictive. For the upcoming analysis, it will be useful to introduce the following notation for the maximal achievable (cumulative) flow
\begin{equation*}
\begin{array}{rlll}
\Phi^*_k(t) :=& \text{max} & \Phi_k(t) & \\ [2ex]
	& \text{subject to} & \Phi(\tau+1) = f(\Phi(\tau),R(\tau)) & \forall \tau \in \{0,\dots,t-1\} \\ [1ex]
	& & R_k(\tau) \leq R_k(\tau+1) \leq R_k(\tau) + \bar r_k & \forall \tau \in \{0,\dots,t-1\}  \\ [1ex]
	& & W_k(\tau) - \bar q_k \leq R_k(\tau) \leq W_k(\tau) & \forall \tau \in \{0,\dots,t-1\} \\ [1ex]
	& & \text{Initial state $\Phi(0)$ and $R(0) = \bf{0}$ given.} & ,
\end{array}
\end{equation*}
at a particular time $t$ and cell $k$.

\begin{lemma} 
\label{lemma:one_step_optimality}
Assume that $\Phi_j(t) = \Phi^*_j(t)$ for $j \in \{k-1,k,k+1\}$ and that the cells $k$ and $k+1$ are nonrestrictive at time $t$. Then $\Phi_k(t+1) = \Phi^*_k(t+1)$. 
\end{lemma}
\begin{proof} 
For notational convenience, let us introduce a relaxation of the set of feasible metering rates
\begin{align}
\mathcal{R}_k(t) := \left\{ R_k : W_k(t) - \bar q_k \leq R_k \leq W_k(t) \right\} \label{eq:relaxed_rampbounds}
\end{align}
 at time $t$. Note that according to this definition, the metering rates satisfy the constraints arising from the limited space on the onramps \eqref{eq:rampbounds2}, but not (necessarily) the additional, constant bounds \eqref{eq:rampbounds1}. Also note that the feasible sets for the metering rates $\mathcal{R}_k(t)$ are decoupled in time in the cumulative formulation, i.e.,\ they do not depend on the system state or previous actions. In the following, we will also use the shorthand notation $\mathcal{R}(t) := \otimes_{k = 1}^{n} \mathcal{R}_k(t)$.
Now, we can write the maximal cumulative flows as
\begin{equation*}
\begin{array}{rlll}
\Phi^*_k(t+1) \leq& \text{max} & f_k(\Phi(t),R(t)) & \\ [1ex]
	& \text{subject to} & \Phi(\tau+1) = f(\Phi(\tau),R(\tau)) & \forall \tau \in \{0,\dots,t-1\} \\ [1ex]
	& & W_k(\tau) - \bar q_k \leq R_k(\tau) \leq W_k(\tau) & \forall \tau \in \{0,\dots,t-1\} \\ [1ex]
	& & \text{Initial state $\Phi(0)$ and $R(0) = \bf{0}$ given.} \\ [3ex]
	=& \max_{R \in \mathcal{R}(t)} & f_k (\Phi^*(t),R) .
\end{array}
\end{equation*}
Note that we only have an inequality in the first step, since we operate with a relaxation of the feasible set of inputs. This ensures that the constraints become decoupled in time. Then, we use monotonicity of $\Phi_k(t+1) = f_k(\Phi(t),R(t))$ in the cumulative flows $\Phi(t)$ to obtain the final equality.  Using this preliminary result and the fact that cells $k$ and $k+1$ are nonrestrictive at time $t$, we will now show that also $\Phi_k(t+1) \geq \Phi_k^*(t+1)$ and hence $\Phi_k(t+1) = \Phi_k^*(t+1)$. 

Cell $k$ is nonrestrictive by assumption and therefore, it satisfies either 
\begin{itemize}
\item[(A)] $d_k(\rho_k(t)) > \phi_k(t) ~\vee~ \phi_{k}(t) = F_k$ or
\item[(B)] $q_k(t) = 0$
\end{itemize}
or both, according to Definition \ref{def:local_optimality}. Similarly, cell $k+1$ is also nonrestrictive and satisfies either 
\begin{itemize}
\item[(C)] $s_{k+1}(\rho_{k+1}(t)) > \phi_k(t) ~\vee~ \phi_k(t) = F_{k}$ or 
\item[(D)] $q_{k+1}(t) = \bar q_{k+1}$
\end{itemize}
or both. We proceed with a case distinction, depending on which of these conditions are satisfied for cell $k$ and cell $k+1$, respectively:

 \begin{itemize} 
\item[(i)] Assume (A) and (C) hold, that is, the flow $\phi_k(t)$ is neither limited by submaximal supply or submaximal demand. 
It follows that $\phi_k(t) = \min \left\{d_{k} (\rho_k), F_k, s_{k+1} \left(\rho_{k+1} \right) \right\} = F_k$ and hence
\begin{align*}
\Phi_k(t+1) &= \Phi_k^*(t) + \Delta t~\cdot F_k \\
 &\geq \max_{R \in \mathcal{R}(t)} f_k(\Phi^*(t),R) \\
 &= \Phi^*_k(t+1) .
\end{align*}
\item[(ii)] Assume (B) and (D) hold, that is, the onramp in cell $k$ is empty 
 and the onramp in cell $k+1$ is full
 . The first condition implies that $R_k(t) = W_k(t)$, according to equation \eqref{eq:agg_rampbound2}. Conversely, the second condition implies that $R_{k+1}(t) = W_{k+1}(t) - \bar q_{k+1}$. It follows that 
\begin{align*}
\Phi_k(t+1) &= \Phi^*_k(t) + \Delta t~\cdot \min \left\{d_{k} (\rho_k),  s_{k+1} \left(\rho_{k+1} \right) \right\} \\
 &= \Phi^*_k(t) + \Delta t~\cdot \min \left\{d_{k} \Big( \rho_k \big( \Phi^*(t),W_{k}(t) \big) \Big), s_{k+1} \Big( \rho_{k+1} \big( \Phi^*(t), W_{k+1}(t) - \bar q_{k+1} \big) \Big) \right\} \\
 &= \Phi^*_k(t) + \Delta t~\cdot \min \left\{d_{k} \Big( \rho_k \big( \Phi^*(t), \max_{R \in \mathcal{R}_k(t) } R \big) \Big), s_{k+1} \Big(\rho_{k+1} \big( \Phi^*(t), \min_{R \in \mathcal{R}_{k+1}(t)} R \big) \Big) \right\} \\
 &\geq \max_{R \in \mathcal{R}(t)} f_k (\Phi^*(t),R) \\
 &= \Phi^*_k(t+1) .
\end{align*}
Here, we use the lower and upper bounds on the cumulative ramp metering rates $R_k(t)$ according to the definition of $\mathcal{R}_k(t)$ in equation \eqref{eq:relaxed_rampbounds}. Finally, the inequality holds because the cumulative flow $\Phi_k(t+1)$ function is nondecreasing in $R_k(t)$ and nonincreasing in $R_{k+1}(t)$, according to Lemma \ref{prop:monotonic_in_inputs}.
\item[(iii,iv)] In the third case, assume that  (A) and (D) hold, 
whereas in the fourth case, assume that (B) and (C) hold. 
The derivations to show that $\Phi_k(t+1) \geq \Phi^*_k(t+1)$ can easily be constructed by combining parts from cases (i) and (ii).
\end{itemize}
All cases compatible with Definition \ref{def:local_optimality} have been verified and the proof has thus been completed.
\end{proof}

A maximization of TDT leads to a minimization of the total discharge flows at every time instant, which in turn corresponds to a minimization of TTS over the whole horizon, as we will show next.

\begin{theorem} 
\label{theorem:main}
Assume that every cell $k \in \{1,\dots,n\}$ of a freeway is nonrestrictive, for the entire horizon $t \in \{1,\dots,T-1\}$. Then TTS is minimized.
\end{theorem} 
\begin{proof} 
The initial conditions are assumed to be fixed, so $\Phi_k^*(0) = \Phi_k(0)$, for all $k$. Because Definition~\ref{def:local_optimality} holds for every cell in every time step, we can apply Lemma \ref{lemma:one_step_optimality} for every cell and $t=1$, yielding $\Phi_k(1) = \Phi_k^*(1)$ for all $k$. Employing induction, we can proceed in the same manner to show that $\Phi_k(t) = \Phi_k^*(t)$ over the complete control horizon. Thus, the cumulative flows are maximized jointly for every cell and every step.

It remains to be shown that joint maximization of all cumulative flows implies minimization of TTS. To do so, the time spent on the mainline
\begin{align*}
\Delta t \cdot \sum_{k = 1}^n l_k \rho_k(t) = \Delta t \cdot \left( \Phi_0(t) - \Phi_{n}(t) + \sum_{k=1}^{n} \left( R_k(t) - \frac{\beta_k}{\bar \beta_k} \Phi_k(t) \right) \right)
\end{align*}
and the time spent on the metered onramps
\begin{align*}
\Delta t \cdot \sum_{k = 1}^n q_k(t) = \Delta t \cdot \sum_{k=1}^n \left( W_k(t) - R_k(t) \right)
\end{align*}
have to be expressed in terms of the cumulative variables. If we sum according to the definition of TTS, we find
\begin{align*}
\TTS &= \Delta t \cdot \sum_{t=0}^{T} \left( \Phi_0(t) - \Phi_{n}(t) + \sum_{k=1}^{n} \left( W_k(t) - \frac{\beta_k}{\bar \beta_k} \Phi_k(t) \right) \right) .
\end{align*}
Analyzing the sign of the coefficients yields the desired result
\begin{align*}
\text{minimize}~ \TTS &=  \Delta t~\cdot \sum_{t=0}^{T} \Bigg( \Phi_0(t) - \max \left\{ \Phi_{n}(t) \right\} + \sum_{k=1}^{n} \bigg( W_k - \underbrace{ \frac{\beta_k}{\bar \beta_k} }_{\geq 0} \max \left\{ \Phi_k(t) \right\} \bigg) \Bigg) \\
 &=  \Delta t~\cdot \sum_{t=0}^{T} \left( W_0(t) - \Phi^*_{n}(t) + \sum_{k=1}^{n} \left( W_k(t) - \frac{\beta_k}{\bar \beta_k} \Phi^*_k(t) \right) \right) .
\end{align*}
All minimizations and maximizations in the previous equations are to be understood as an optimization over the set of feasible ramp metering patterns, with respect to the fixed initial condition and traffic demands and the CTM system model. Note that $\bar\beta_k \in (0,1]$ for all $k$ by definition of the split ratios.
The last equality confirms that TTS is indeed minimized for $\Phi_k(t) = \Phi^*_k(t)$ in case of a nonrestrictive freeway and thus concludes the proof.
\end{proof}

It is important to keep in mind that Definition \ref{def:local_optimality} provides only \emph{sufficient} optimality conditions. In particular, it is possible that, depending on the freeway parameters and the external demand profile, there does not exist a policy which satisfies the stated optimality conditions. It also becomes clear that the difficulties in solving the minimal TTS problem \eqref{eq:main_problem} hinge mainly on the constant metering bounds \eqref{eq:const_bounds}: recall that the range of admissible ramp metering rates $r_k(t)$ is given by
\begin{align*}
\min \left\{ 0, \frac{1}{\Delta t} \left(\bar q_k - q_k(t) \right) + w_k(t) \right\} \leq r_k(t) \leq \min \left\{ \bar{r}_k , \frac{1}{\Delta t} q_k(t) + w_k(t) \right\} ,
\end{align*}
i.e.,\ the metering rate is constrained by constant bounds and by bounds that depend on the onramp queue length. For the purpose of deriving a lower bound, one can formulate a relaxed min-TTS problem, in which the constant bounds on the ramp metering rates are removed.
\begin{corollary}
\label{corollary}
Consider the problem of minimizing TTS \eqref{eq:main_problem} \emph{without} the constant bounds on the metering rates \eqref{eq:const_bounds}. Simulate the system with metering rates chosen according to the feedback law
\begin{align*}
 r_k(t) = \left[ \frac{l_k}{\Delta t} \left( \rho^c_k - \rho_k(t) \right) + \frac{\phi_k(t)}{\bar\beta_k} - \phi_{k-1}(t) \right]_{ \frac{1}{\Delta t} q_k(t) + d_k(t) }^{ \frac{1}{\Delta t} \left( q_k(t) - \bar q_k \right) + d_k(t) } ~,
 \end{align*}
which we call the \emph{relaxed} BE controller, for all $k \in \{1,\dots,n\}$ and all $t \in \{0,\dots,T\}$. Then, the result of this simulation is a solution to the relaxed optimization problem and the corresponding cost $\TTS_\text{LB}$ is a valid lower bound on the optimal value $\TTS^*$ of the original problem \eqref{eq:main_problem}, that is, $\TTS_{\text{LB}} \leq \TTS^*$.
\end{corollary}
Note that optimality of the relaxed BE controller (for the relaxed problem) follows immediately from Theorem \ref{theorem:main} and Lemma \ref{lemma:feedback_optimality}. 
The feedback policy defined in Corollary \ref{corollary} is potentially infeasible for the original problem, hence its purpose is only the computation of a lower bound. Conversely, an upper bound on the optimal TTS can be computed efficiently by simulating the best-effort controller, which respects all bounds on the metering rates. Together, one can obtain a-posteriori bounds  $\TTS_{\text{LB}} \leq \TTS^* \leq \TTS_{\text{BE}}$ for any given external traffic demand profile via two computationally inexpensive forward simulations. These bounds will be useful for the evaluation in the following section, for cases where the optimality conditions based on nonrestrictivness are not satisfied.

\section{Application} 
\label{sec:application}

\begin{figure}[t] 
	\centering
	\begin{subfigure}[a]{0.34\textwidth}
	\vspace{-5.6cm}
		\includegraphics[width=\textwidth]{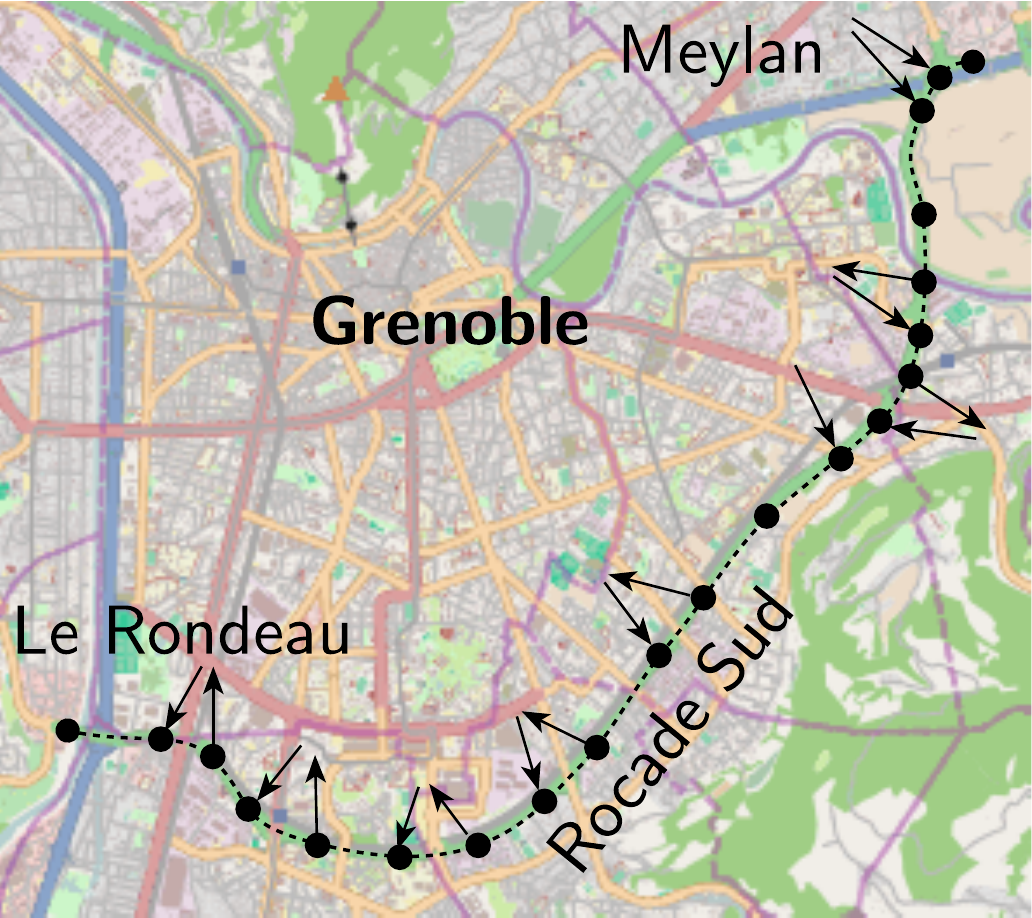}
		\caption{Freeway topology of the Rocade Sud. Map data \copyright2016 Open Street Maps}
	\end{subfigure} \quad
	\begin{subfigure}[b]{0.3\textwidth}
		\includegraphics[width=\textwidth]{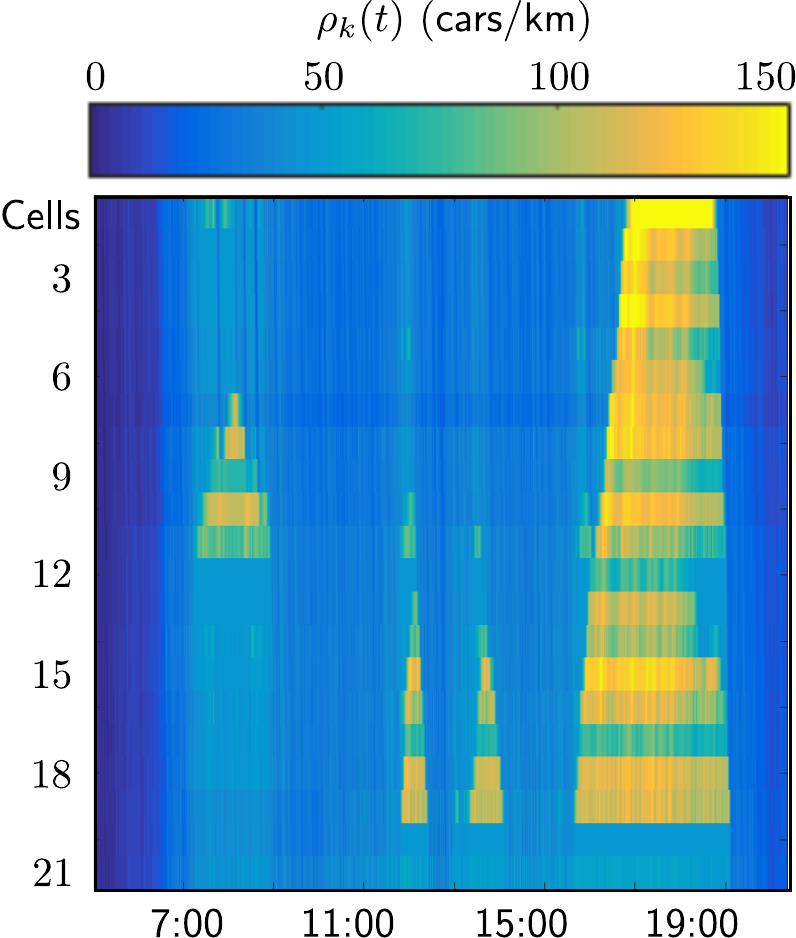}
		\caption{Density without ramp metering}
	\end{subfigure} \quad
	\begin{subfigure}[c]{0.3\textwidth}
	\vspace{-6.25cm}
		\includegraphics[width=\textwidth]{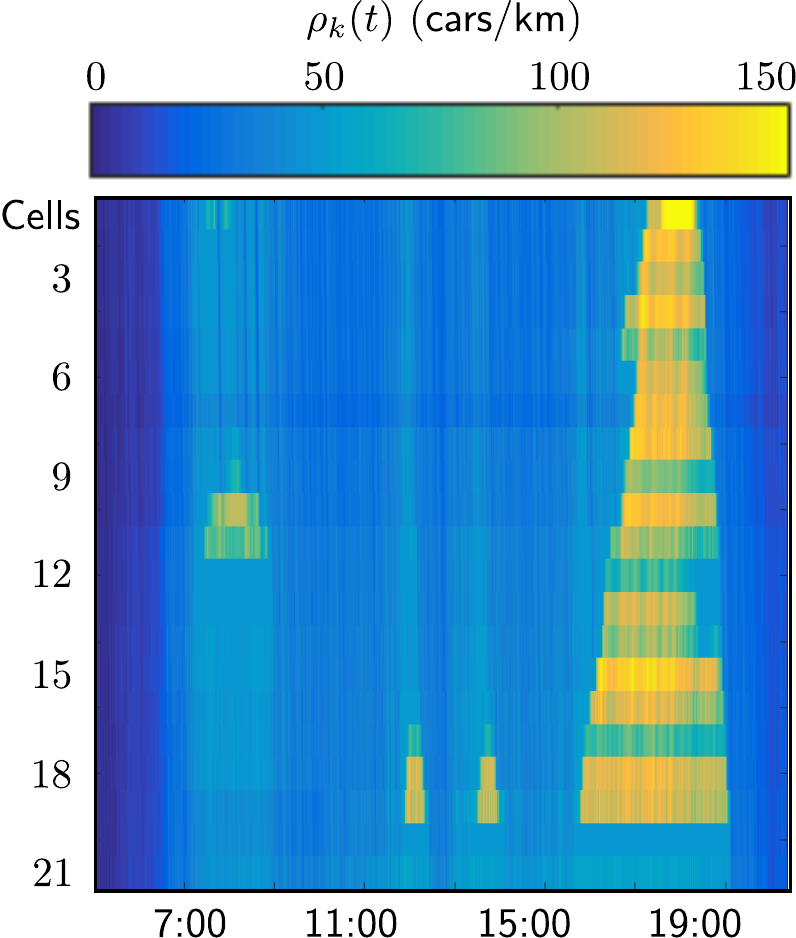}
		\caption{Density with optimal control}
	\end{subfigure} 
	\caption{An overview of the topology of the considered freeway, the Rocade Sud in Grenoble, is depicted on the left. In the middle and on the right, partial simulation results traffic demands of April, 14th, 2014 are shown. Depicted are the mainline densities simulated on the monotonic freeway without any disturbances. The improvement achieved via ramp metering is clearly visible, but even in case of perfect model knowledge and perfect demand prediction it is not possible to prevent congestion altogether.}
	\label{fig:topology}
\end{figure}

We consider a congestion-prone freeway in the vicinity of Grenoble, France \cite{de2015grenoble}, as depicted in Figure \ref{fig:topology}. The total length of this freeway stretch, the so-called Rocade Sud, is $11.8$km. The freeway has $7$ offramps and $10$ onramps in total, some of which are planned to be used for ramp metering in the future. The freeway has been equipped with loop detectors and traffic data are reported to a control center every $15$ seconds. In this case study, the freeway is modeled by the CTM, with one cell per mainline sensor location. A triangular fundamental diagram with parameters as reported in \ref{sec:parameters} is chosen. Since ramp metering has not yet been installed, we can obtain the traffic demand profiles from the onramp- and mainline inflows to the system. For the following case study, we consider data from 5 weeks in March, April and June 2014, corresponding to 35 days in total\footnote{Data for the full three months are available, however, data from many days are incomplete due to sensor failures. To ensure representative performance evaluation, only weeks for which data are complete for all sensors at all days of the respective week are chosen for the case-study. }. We are interested in two main questions:
\begin{itemize}
\item How does TTS for best-effort control compare to the optimal TTS, in a theoretical best-case scenario with perfect model knowledge and traffic demand prediction?
\item Can we use best-effort control respectively the optimal solution to gain insight into the performance achieved by practical ramp metering controllers (Alinea) in a realistic scenario, that is, in the presence of model and demand uncertainty?
\end{itemize}
To answer the first question, we will study simulations of the nominal model in the following section, before considering the effects of model uncertainty and random disturbances of external demands in Section \ref{sec:empirical}.

\subsection{Nominal case}  
\label{sec:nominal}

\begin{figure}[t] 
	\centering
		\includegraphics[width=0.99\textwidth]{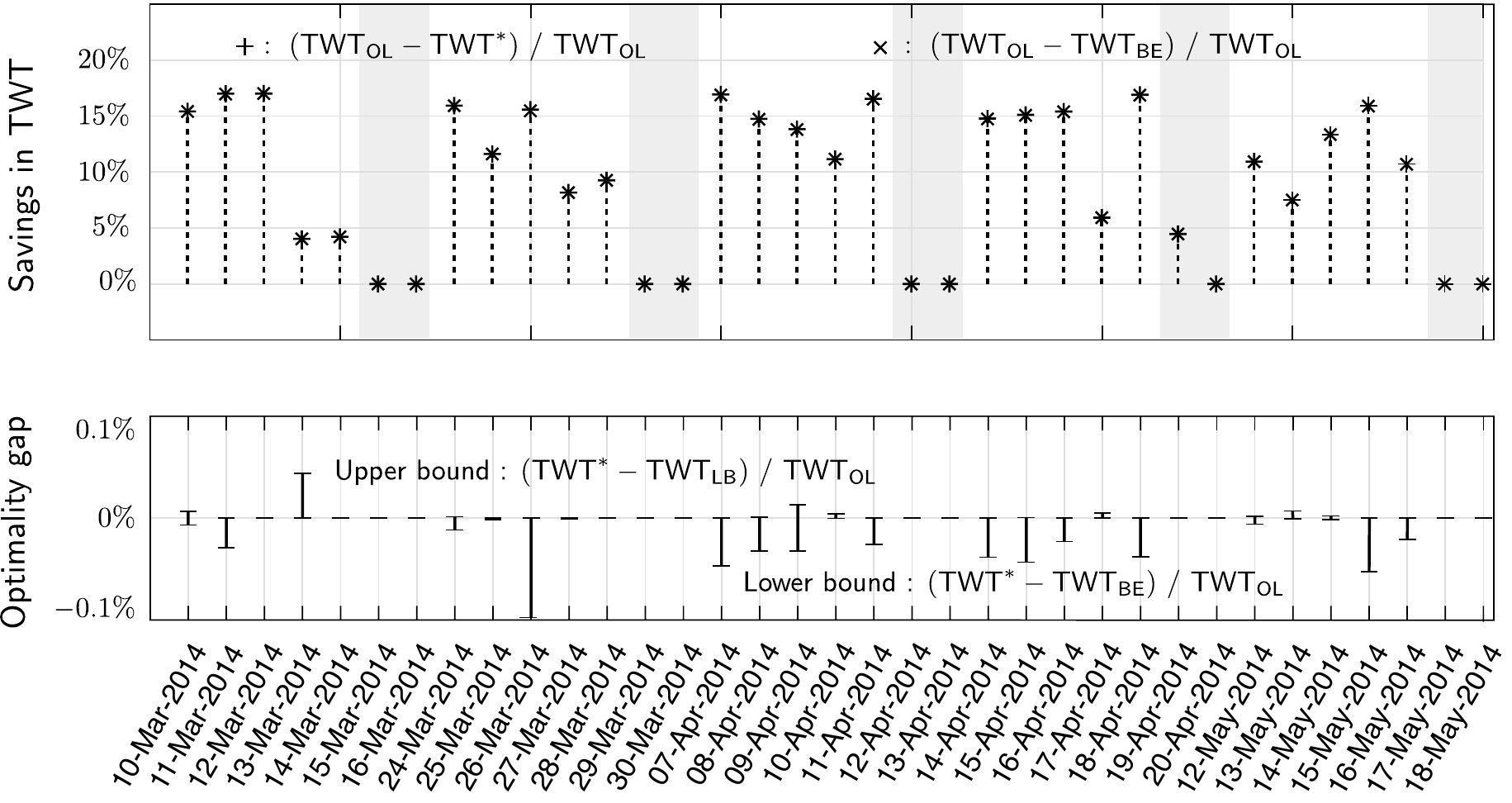}
	\caption{The savings in TWT for best-effort control are almost identical to the optimal savings, as it can be seen in the upper plot. Note that on weekends, highlighted in grey, ramp metering does provide little or no benefits. In the second plot, we chose a much smaller scale to visualize the suboptimality of best-effort control. The optimal savings can be lower bounded by the best-effort solution and upper bounded by the relaxed best-effort solution $\TWT_{\text{LB}}$, computed as described in the previous section. It can be seen that for the chosen scenarios, the suboptimality of best-effort control is less than $0.1\%$ (in savings of TWT).}
	\label{fig:case_study_hist}
\end{figure}

In this section, we aim to compare the performance of best-effort control, that is, the Total Time Spent $\TTS_{\text{BE}}$ to the optimal solution $\TTS^*$ and the Total Time Spent in open loop $\TTS_{\text{OL}}$, that is, without any ramp metering . The optimal solution can be obtained by solving a Linear Program as detailed in \cite{gomes2006optimal}, but it is worth highlighting that its computation requires perfect traffic demand prediction. Furthermore, both best-effort control and the computation of the optimal solution require a traffic model. For now, we assume that the freeway dynamics are given by the nominal model as described before and that this model is known to the ramp metering controllers. The aim is to verify the theoretical results and the assumption of perfect model knowledge will be relaxed in Section \ref{sec:empirical}.

We simulate the traffic evolution using the nominal traffic model and the recorded external traffic demand patterns of five weeks in total. Congestion patterns for a day with large, but not atypical external traffic demands (14.04.2014) are depicted in Figure \ref{fig:topology}. Both the evolution of the uncontrolled freeway and the optimal solution are shown, in order to provide some intuition on the typical benefits achievable by ramp metering. A natural metric to quantify the benefits of ramp metering is to relate the savings in TTS to TTS itself.  On average (over all days), we obtain savings of
\begin{equation*}
\frac{\TTS_{\text{OL}} - \TTS^*}{\TTS_{\text{OL}}} \approx 3.09\% .
\end{equation*}
A large proportion of total TTS is caused by vehicles traveling at free-flow velocity, for which ramp metering does not provide any benefits. We are mainly interested in comparing the total time wasted in congestion and on onramp queues, however, which we call the \emph{Total Waiting Time} (TWT). Hence, we first define the Total Free-flow Time (TFT) as the TTS achieved on a hypothetical freeway in which all cars instantly enter the mainline after arriving on an onramp and always travel at free-flow speed on the freeway itself. Then, the total waiting time can be computed as $\TWT^* = \TTS^* - \TFT$ (and similarly for $\TWT_{\text{OL}}$) and the relative savings in terms of time wasted in congestion and on onramp queues, over all days, amount to
\begin{equation*}
\frac{\TWT_{\text{OL}} - \TWT^*}{\TWT_{\text{OL}}} \approx 14.6\% .
\end{equation*}
We note that these numbers are similar to the ones reported in \cite{gomes2006optimal,muralidharan2012optimal}. 

The complete results for all individual days are summarized in Figure \ref{fig:case_study_hist}. It can immediately be seen that ramp metering indeed reduces TTS on most days. In the monotonic CTM, this is achieved by a reduction in the spill back of mainline congestion, which in turn increases the outflows from the offramps. On certain days, no improvement can be achieved by ramp metering. These days are characterized by low traffic demand, which is for example typical for weekends. From real traffic data, one can verify that even the uncontrolled freeway usually does not become congested on these days, so obviously no ramp metering is the best policy. It is noteworthy that the performance achieved by the best-effort controller is indistinguishable from the optimal solution in a plot scaled according to the absolute values of the savings. A closer look reveals that the worst performance deterioration of the best-effort controller amounts to
\begin{align*}
\max_{d \in \text{``days"}} \left\{ \frac{\TWT^*(d) - \TWT_{\text{BE}}(d)}{\TWT_{\text{OL}}(d)} \right\} \approx -0.1 \%
\end{align*}
in terms of the time wasted in congestion (March 26th). In comparison to the savings achieved by best-effort control, we consider this optimality gap to be negligible. As further evidence of the value of our results, we note the observation of \cite{muralidharan2012optimal} who show that increasing the length of the control horizon in MPC for ramp metering does not seem to substantially improve performance\footnote{According to \cite[Table II]{muralidharan2012optimal}, extending the control horizon from $5$min to $25$min leads only to a marginal improvement in terms of savings of TWT from $17.39\%$ to $17.43\%$.}. This at first sight counterintuitive observation can be interpreted through our results, which suggest that the BE controller, which is equivalent to MPC with a horizon of a single time step, performs close to optimal for ramp metering of a freeway stretch modeled by the monotonic CTM.

Nevertheless, the existence of a (small) optimality gap for the best-effort controller suggests that at least some cells become restrictive according to Definition \ref{def:local_optimality} at one point during most days. Analysis of the simulation results reveals that most cells are nonrestrictive for most of the time. More precisely, the day on which restrictive cells occur most often  is April 18th, on which any cell corresponding to a metered onramp is restrictive on average in $0.39\%$ of all sampling time instances. 
March 12th, March 14th and March 28th are noteworthy special cases, since all cells are nonrestrictive all the time in simulations for these days, despite the occurrence of congestion. Therefore, these days provide an opportunity to verify the results of Theorem \ref{theorem:main}. Indeed, we can verify in Figure \ref{fig:case_study_hist} that best-effort control achieves exactly the optimal performance on these days.

Theorem \ref{theorem:main} provides only sufficient (but not necessary) conditions for optimality. Also, it does not make any a-priori statement about performance if some cells become restrictive, as it happens in the empirical evaluation for most days. The simulation results suggest, however, that infrequent violations of the optimality conditions lead to solutions that are close to optimal. Therefore, it seems worthwhile to analyze when cells become restrictive. Analysis of the results reveals that of all cells, cell $7$ becomes restrictive the most often by far. This effect can partially be explained by the freeway topology in the vicinity of cell $7$. As depicted in Figure \ref{fig:cell7_map}, the traffic demand at the onramp at cell $7$ originates from the same roundabout as the traffic demand at the next downstream onramp (cell $8$). Drivers seem to prefer the downstream onramp and therefore, the traffic demand at the onramp corresponding to cell $7$ is low in comparison to other onramps on the freeway. Thus in case of ramp metering, it takes longer to fill up the onramp queue and the storage space on the respective onramp is not used to its best potential during short time intervals. During these intervals, cell $7$ becomes restrictive as depicted in Figure \ref{fig:cell7}. This effect is the main cause for the (still negligible) suboptimality of best-effort control in this simulation study.

\begin{figure}[t] 
	\centering
	\begin{subfigure}[a]{0.2\textwidth}
		\centering
		\includegraphics[width=\textwidth]{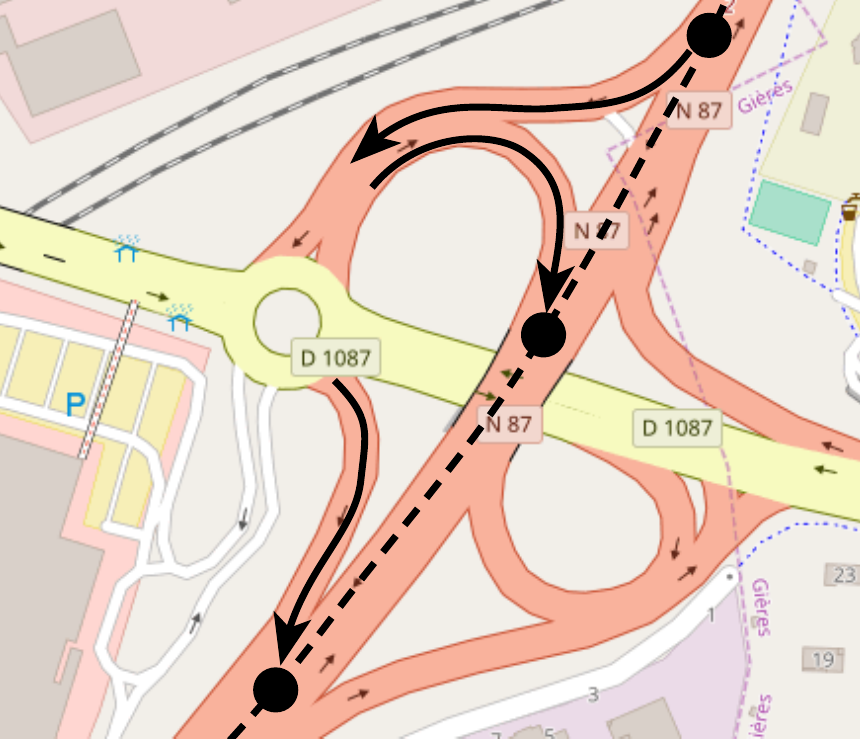}
		\caption{Map of cells $7$ and $8$ and the associated onramps of the \emph{Rocade Sud}, Grenoble. Map data \copyright2017 Open Street Maps.}
		\label{fig:cell7_map}
	\end{subfigure} ~
	\begin{subfigure}[a]{0.37\textwidth}
		\centering
		\includegraphics[width=\textwidth]{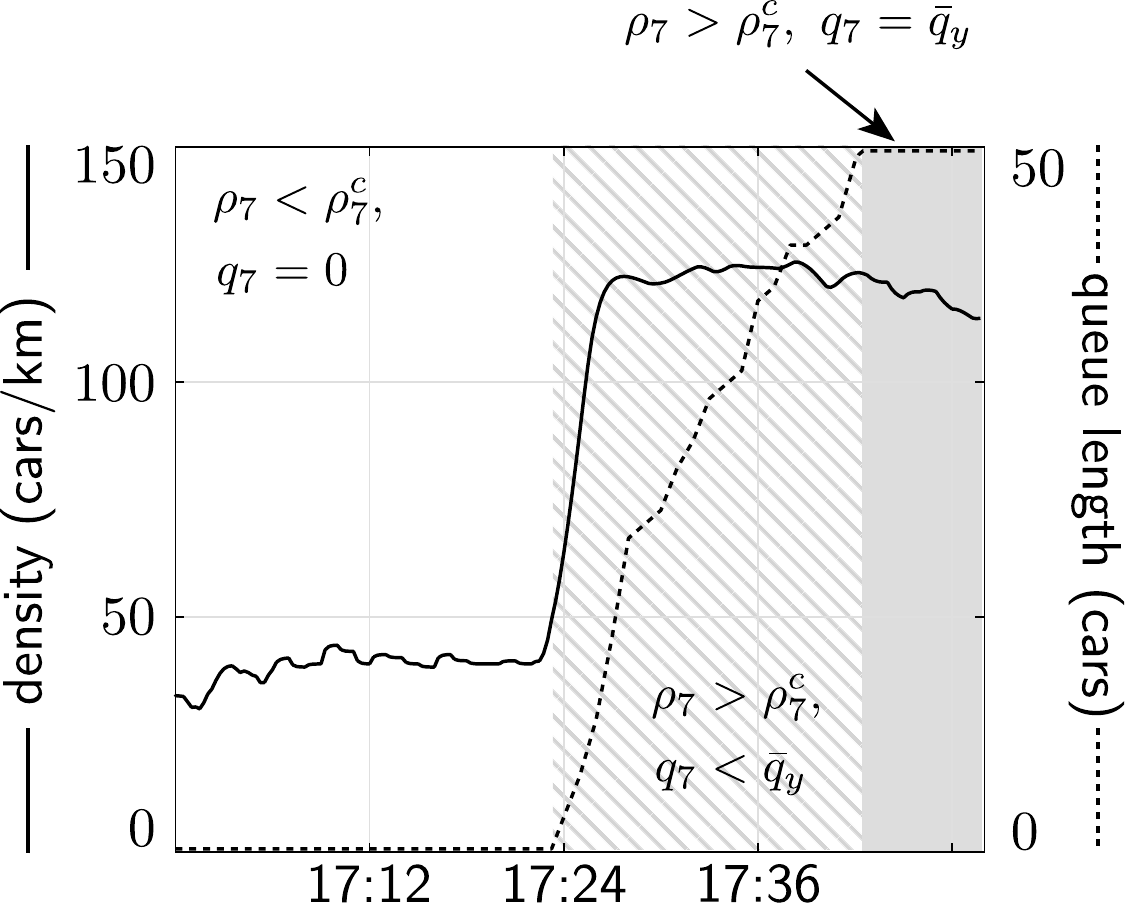}
		\caption{Density evolution in cell $7$ at the onset of evening congestion. During a $18$min period, cell $7$ is restrictive since it is congested while the associated onramp is not full.}
		\label{fig:cell7}
	\end{subfigure} ~	
	\begin{subfigure}[a]{0.37\textwidth}
		\centering
		\includegraphics[width=\textwidth]{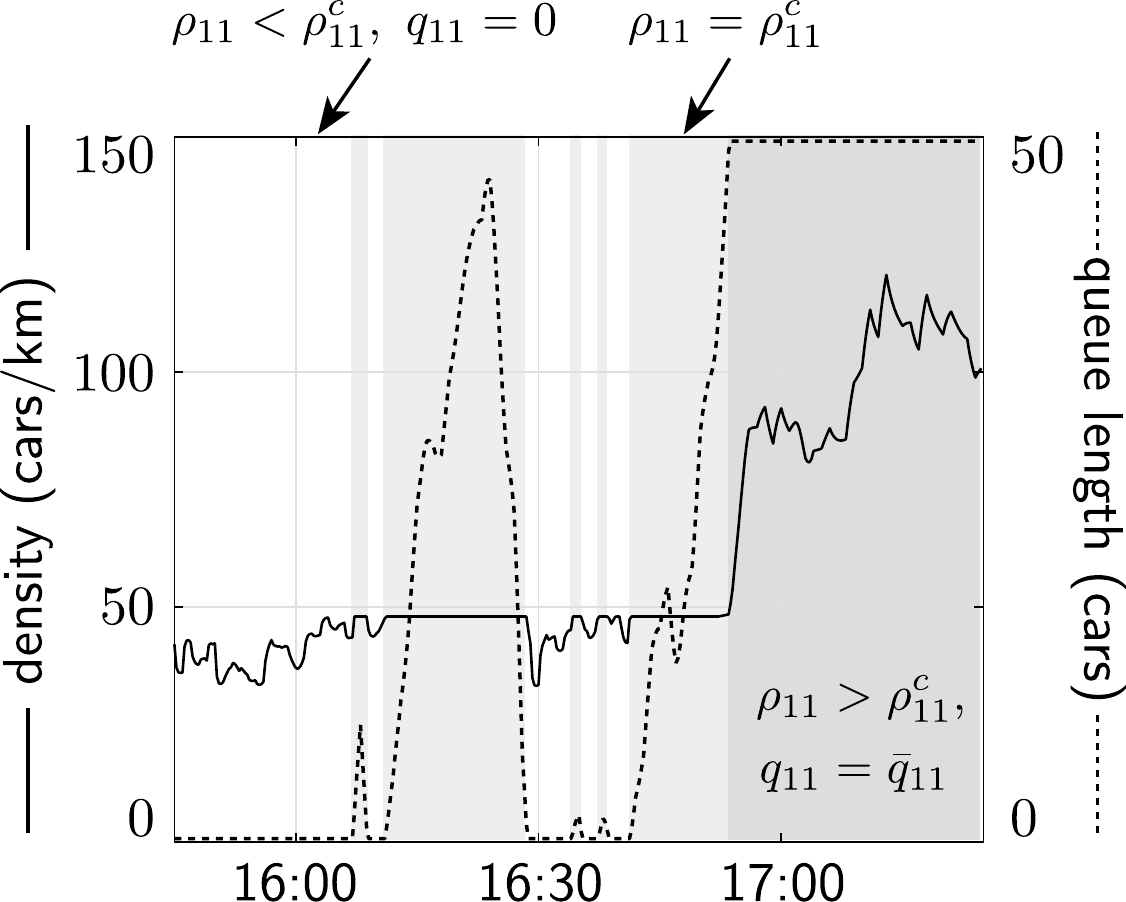}
		\caption{Density evolution in cell $11$ at the onset of evening congestion. This cell is nonrestrictive during the whole period.}
		\label{fig:cell11}
	\end{subfigure} ~
	\caption{Partial simulation results of the \emph{Rocade Sud}, using the traffic demand of April 14th, 2014. Cell $7$ is atypical, since traffic demands at the associated onramp are very low because the onramp is connected to the same roundabout as onramp $8$. Therefore, it typically takes a long time to fill the ramp which facilitates the corresponding cell becoming restrictive.}
\end{figure}

\begin{figure}[tp] 
	\centering
		\includegraphics[width=0.7\textwidth]{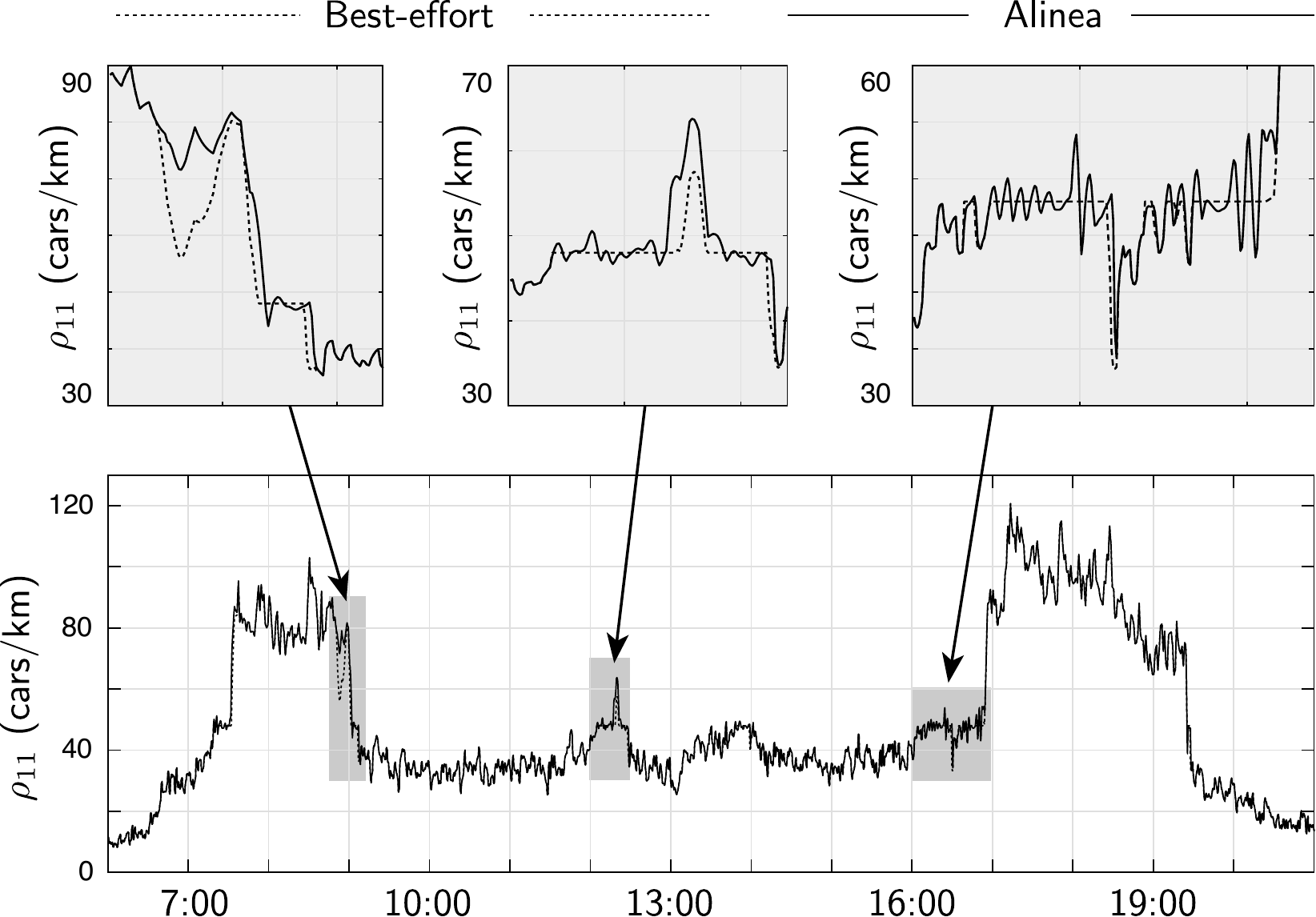}
	\caption{Comparison between the closed-loop trajectories of best-effort control and ALINEA for cell $20$ on April 14th, 2014.}
	\label{fig:alinea_detail}
\end{figure} 

To demonstrate that the theoretical conclusions drawn from the analysis of the best-effort controller indeed extend to practical ramp metering policies, we perform a comparison to ALINEA \cite{papageorgiou1991alinea}, a popular, distributed ramp metering policy. ALINEA in its basic form consists of local, anti-windup integral feedback controllers. The metering rates are first computed as integral feedback
\begin{equation*}
\tilde r_k(t) := r_k(t-1) + K_I \cdot (\rho^c_k - \rho_k(t))
\end{equation*}
Then, they are saturated
\begin{align}
r_k(t) = \left[ \tilde r_k(t) \right]_{\max \left\{ 0, \frac{1}{\Delta t} \left( q_k(t) - \bar q_k \right) + d_k(t) \right\} }^{\min \left\{ \bar r_k, \frac{1}{\Delta t} q_k(t) + d_k(t) \right\} }
\end{align}
 in the same way as for the BE controller. There exists a variety of extensions to this basic controller \cite{papamichail2008traffic,smaragdis2004flow,wang2013local}, that introduce coordination between ramps or permit use of different sensor configurations. The standard ALINEA controller requires only the critical density as a model parameter. Instead of using model knowledge to estimate and predict split ratios and traffic demand and supply, this controller relies on integral feedback. Closed loop equilibria of Alinea in the monotonic CTM are known to be flow optimal \cite{schmitt2015flow}.

Comparing the optimal solution $\TWT^*$ reps. $\TWT_{\text{BE}}$ to the Total Waiting Time achieved by Alinea $\TWT_{\text{AL}}$, we find that the performance deteriorates only slightly, with an average loss of
\begin{equation*}
\frac{\TWT_{\text{AL}} - \TWT^*}{\TWT_{\text{OL}}} = 0.45\%, \quad \frac{\TWT_{\text{AL}} - \TWT_{\text{BE}}}{\TWT_{\text{OL}}} = 0.42\%, 
\end{equation*}
in terms of savings in $\TWT$. Closer inspection reveals that Alinea does not only show comparable performance to the best-effort controller (and hence, the optimal solution), but also very similar trajectories, as depicted in Figure \ref{fig:alinea_detail}. This figure shows in detail the density evolution in cell $20$ in simulations for both best-effort control and Alinea, using the external traffic demands of April 14th, 2014. Cell $20$ is a major bottleneck and the differences between both controllers are most pronounced here. Nevertheless, deviations in the trajectories obtained by the respective controllers are small in comparison to the variations caused by time-varying external traffic demands. Differences only occur in time intervals during which the controllers do not saturate, i.e.,\ when the density is stabilized at (or close to) the critical density. Changes in upstream- and downstream mainline flows act as persistent disturbances and the lack of perfect model knowledge in Alinea means that the density will rarely, if ever, converge exactly to the critical density. By contrast, the best-effort controller is assumed to have perfect model knowledge and keeps the density exactly at the critical density if the metering bounds permit to do so. In this sense, one should view the slight performance deterioration of ALINEA in comparison to best-effort control as the price one has to pay for not exactly knowing the fundamental diagram in reality. Nevertheless, the similarity between the trajectories suggests that (ideal) best-effort control is a suitable proxy to gain insight into the performance of Alinea.

\subsection{Model Uncertainty}  
\label{sec:empirical}

We now drop the assumption of perfect model knowledge in order to test the influence of model uncertainty on the controller performance. In addition to the nominal freeway model as introduced in the preceding section, we also consider 
\begin{itemize}
 	\item \textbf{(Model uncertainty)} the effect of a controller model with a fundamental diagram different from the nominal model used for simulations,
	\item \textbf{(Disturbances)} the effect of random, uncorrelated disturbances acting on the flows during simulation and 
	\item \textbf{(Capacity drop}) the effects of a non-monotonic demand function which models a capacity drop in the congested region.
\end{itemize}
Specifically, we sample (random) controller models for the fundamental diagram as follows: the controller model uses free-flow speeds $\hat v_k$ drawn from a uniform distribution centered around the nominal free-flow speeds $\hat v_k \sim \mathcal{U}(v_k - \Delta v_k, v_k + \Delta v_k))$. Similarly, the controller assumes a traffic jam density $\hat {\bar \rho}_k$ randomly chosen from the uniform distribution $\hat {\bar \rho}_k \sim \mathcal{U}(\bar\rho_k - \Delta \bar\rho_k, \bar\rho_k + \Delta \bar\rho_k))$. The resulting range of model uncertainty in the fundamental diagram is depicted in Figure \ref{fig:uncertain_fd}. We do \emph{not} perturb the critical density $\rho_k^c$, as the fact that the ``critical occupancy [proportional to the critical density considered here] seems to be less sensitive with respect to weather conditions and other operational influences compared with the capacity [$F_k$] of a freeway stretch" \cite{papageorgiou1991alinea} is the main reason it is preferably used as a reference by feedback control in ramp metering. Uncertainty in the fundamental diagram is thus modeled in a static manner, but traffic flows might also be affected by random disturbances. Therefore, we assume that the flows $\hat \phi_k$ are subject to time-varying, normally distributed disturbances $\hat \phi_k(t) \sim \phi_k(t) \cdot \mathcal{N}(1, \sigma_\phi)$. Finally, recall that the theoretical analysis in this work applies to the monotonic CTM. We will also extend the empirical analysis to a non-monotonic variation\footnote{Note that in case of the CTM, non-monotonicity of the demand function implies non-monotonicity of the dynamic system in the sense of e.g.\ \cite{hirsch2005monotone2}.} of the fundamental diagram as depicted in Figure \ref{fig:nonmonotone_fd} \cite{kontorinaki2016capacity}.

\begin{figure}[t] 
	\centering
	\begin{subfigure}[a]{0.3\textwidth}
	\vspace{-6.2cm}
		\includegraphics[width=\textwidth]{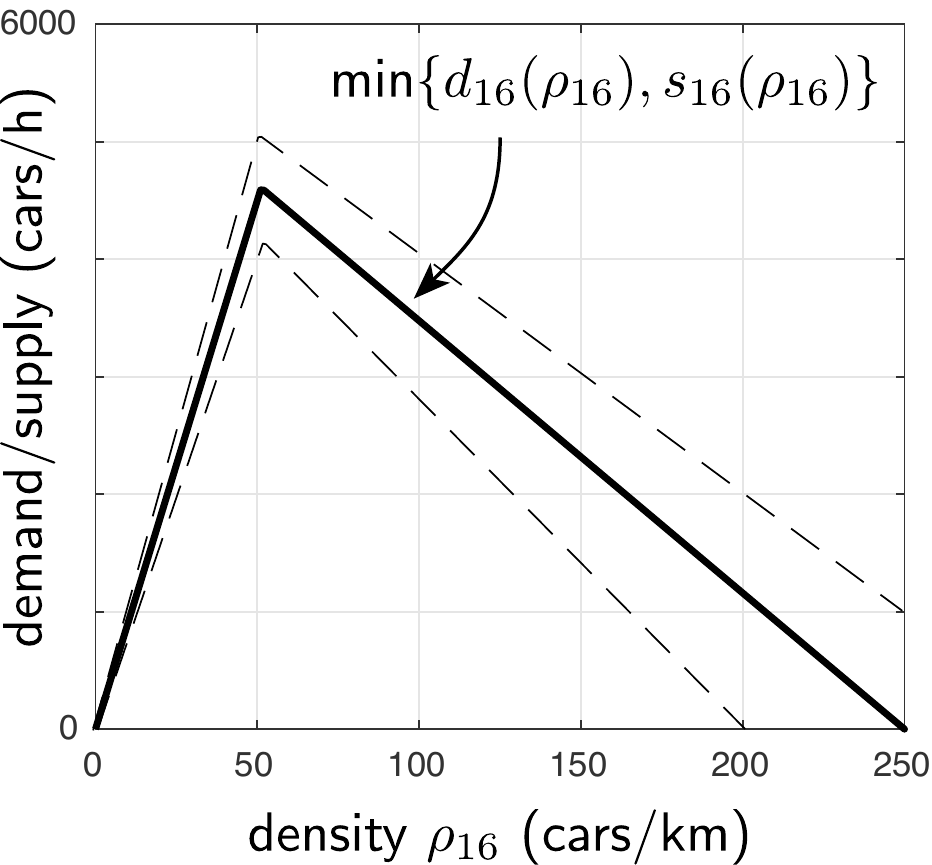}
		\caption{Maximal uncertainty range for $\Delta v = 10\%$ and $\Delta \bar\rho = 20 \%$.}
		\label{fig:uncertain_fd}
	\end{subfigure} ~
	\begin{subfigure}[b]{0.3\textwidth}
		\includegraphics[width=\textwidth]{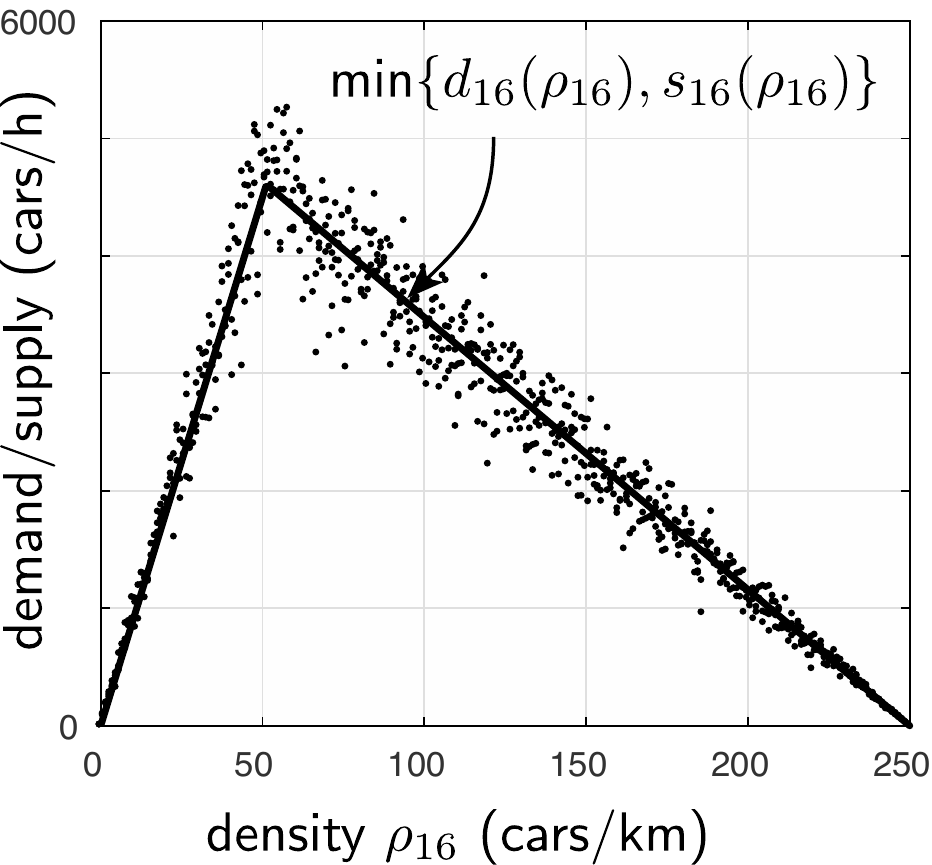}
		\caption{Nominal fundamental diagram in comparison to random flows, for $\sigma = 5\%$. }
		\label{fig:disturbed_fd}
	\end{subfigure} ~
	\begin{subfigure}[c]{0.3\textwidth}
	\vspace{-6.2cm}
		\includegraphics[width=\textwidth]{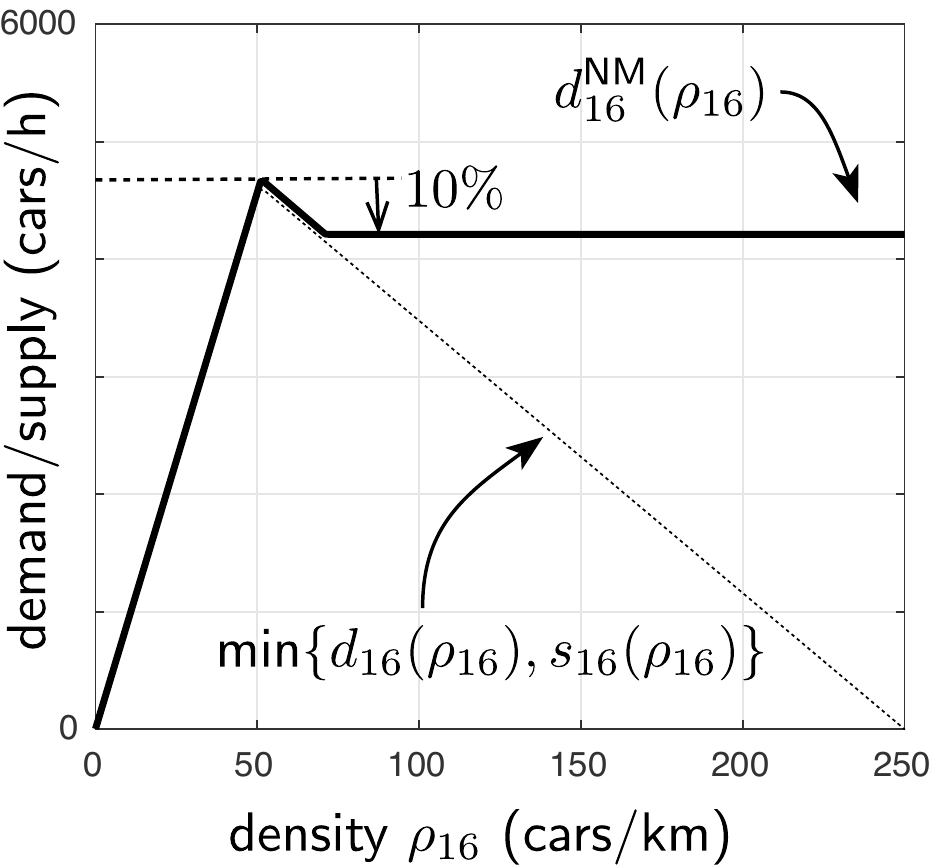}
		\caption{Capacity drop in the demand function in the congested region.}
		\label{fig:nonmonotone_fd}
	\end{subfigure} 
	\caption[TEST]{For cell $16$, the maximal range of uncertainty for the controller model is depicted on the left. In the middle, realizations of the random flows are compared to the nominal fundamental diagram\protect\footnotemark. As depicted on the right, the non-monotonic FD includes a capacity drop of $10\%$ in the demand function.}
\end{figure}
\footnotetext{For illustration purposes, we perturb the local fundamental diagram $\min \{ d_{16}(\rho_{16}), s_{16}(\rho_{16}) \}$. Note that in simulations, the perturbations are applied as $\hat \phi_{16}(t) = \min \{ d_{16}(\rho_{16}(t)), s_{17}(\rho_{17}(t)) \} \cdot \mathcal{N}(1,\sigma_\phi)$.}

\begin{table} 
\centering
\caption{Average improvement in TWT for ramp metering (over the $35$ considered days) in the simulation study. For values depending on stochastic disturbances and model uncertainty, the mean values over $100$ independent simulations are reported. The values from the nominal analysis in the previous chapter are highlighted in \textbf{bold} and reported with an additional digit of accuracy, to highlight that they differ slightly from each other.}
\begin{tabular}{rllcccc}
\toprule 
 \multicolumn{3}{c}{} & \multicolumn{4}{c}{Improvement in TWT} \\ [1ex]
 \cmidrule(lr){4-7}
 & & & \multicolumn{2}{c}{Monotonic CTM} & \multicolumn{2}{c}{Non-monotonic CTM} \\ [1ex]
 \cmidrule(lr){4-5} \cmidrule(lr){6-7}
 & & & $\sigma_\phi = 0\%$ & $\sigma_\phi = 5\%$ & $\sigma_\phi = 0\%$ & $\sigma_\phi = 5\%$  \\ [1ex]
 \cmidrule(lr){4-4} \cmidrule(lr){5-5} \cmidrule(lr){6-6} \cmidrule(lr){7-7}
 & \multicolumn{2}{r}{Optimal LP solution} & $\bf 14.62\%$ & - & - & - \\ [1.5ex]
\multirow{ 4}{*}{\begin{sideways}Best-effort\end{sideways}}
 & $\Delta v = 0\%$,& $\Delta \bar\rho = 0\%$ & $\bf 14.59\%$ & $13.7\%$ & $28.3\%$ & $18.5\%$ \\ [0.7ex]
 & $\Delta v = 2.5\%$,& $\Delta \bar\rho = 5\%$ & $14.2\%$ & $13.4\%$ & $25.8\%$ & $17.9\%$ \\ [0.7ex]
 & $\Delta v = 5\%$,& $\Delta \bar\rho = 10\%$ & $13.7\%$ & $12.9\%$ & $23.4\%$ & $17.3\%$ \\ [0.7ex]
 & $\Delta v = 10\%$,& $\Delta \bar\rho = 20\%$ & $12.4\%$ & $12.2\%$ & $19.6\%$ & $15.8\%$ \\ [1.5ex]
 & \multicolumn{2}{r}{Alinea} & $14.2\%$ & $12.6\%$ & $25.0\%$ & $16.6\%$ \\ [1.5ex]
\bottomrule
\end{tabular}
\label{tab:performance}
\end{table} 

We simulate the best-effort controller and Alinea for different combinations of $\Delta v \in \{0\%, 2.5\%, 5\%, 10\%\}$, $\Delta \bar\rho \in \{0\%, 5\%, 10\%, 20\%\}$, $\sigma_{\phi} \in \{0\%, 5\%\}$ and the CTM either monotonic or with a capacity drop of $10\%$. Note that Alinea is not affected by model uncertainty in either $v$ or $\bar\rho$. We also simulate the freeway in open loop, that is, without ramp metering, and compute the corresponding $\TFT$ in order to compute the savings in TWT. The results are reported in Table \ref{tab:performance}. The case of the monotonic CTM with $\Delta v = \Delta \bar\rho = \sigma = 0\%$ corresponds to the nominal traffic model discussed in detail in the previous section. The optimal solution is only known for this special case. 

Unsurprisingly, the addition of random noise ($\sigma_\phi = 5\%$) decreases the effectiveness of ramp metering. Similarly, increasing model uncertainty deteriorates the performance of best-effort control. It might be less obvious why greater savings in TWT are achieved in cases when the non-monotonic CTM is used for simulations, even though the best-effort controller always uses a monotonic CTM for its one-step-ahead flow predictions. The reason is that in a non-monotonic model, the prevention of congestion via ramp metering provides the additional benefit of increased bottleneck flow, in addition to the increased offramp outflow achieved via prevention of congestion spill back, which is present in both monotonic and non-monotonic models.

Note that the performance achieved by Alinea is bounded by the performance of nominal best-effort control ($\Delta v = \Delta \bar\rho = 0\%$) and best-effort control with substantial, but not unrealistic model uncertainty ($\Delta v = 10\%$, $\Delta \bar\rho = 20\%$) in this simulation study. This further supports the hypothesis that best-effort control is a suitable proxy for studying the performance of Alinea. In case of a monotonic freeway model, this answers the question of \emph{why} Alinea is ``close to the most-efficient logic" \cite{zhang2004optimal}: It can be seen as a practical approximation to the best-effort controller, which by itself is impractical in the absence of accurate model knowledge. The suboptimality of Alinea in the monotonic case is therefore mainly caused by the lack of perfect model knowledge and \emph{not} a consequence of the lack of long-term demand prediction or centralized coordination.

It should be highlighted that this conclusion only holds true for the monotonic case. Even though it can be seen that Alinea approximates best-effort control in a similar manner for the non-monotonic CTM, the best-effort solution is not necessarily close to the optimal solution. In fact, it is not even clear if the optimal solution in the non-monotonic case can be computed efficiently, even in the idealized case with perfect model knowledge, since the resulting optimization problems are non-convex. If useful theoretical bounds on optimal performance in the non-monotonic setting were to become available however, best-effort control might be useful in analyzing the suboptimality of distributed, non-predictive control in a similar manner as in our analysis for the monotonic case.

\section{Conclusions} \label{sec:conclusions} 

In this work, we have derived sufficient optimality conditions for minimal TTS ramp metering, for a freeway modeled by the monotonic CTM. To keep the theoretical analysis tractable, we considered an idealized setting with perfect model knowledge. The results imply that as long as the cells are nonrestrictive, the distributed, non-predictive best-effort controller is optimal and in the monotonic CTM, no additional benefit can be realized either by coordination between ramps or long-term prediction of the external traffic demands. The fact that space on the onramps is limited according to constraint \eqref{eq:rampbounds2} does not make coordination between ramps desirable in the monotonic setting, though it does make ramp metering less effective, of course. Yet, Theorem \ref{theorem:main} does not make any a-priori statement about whether cells will remain nonrestrictive or, in case they do not, about the optimality gap of best-effort control in comparison to the optimal solution. A partial remedy can be achieved by employing the relaxed BE controller to compute a lower bound on TTS. Even though no a-priori performance bound for \emph{all} external demand patterns is available, the methods presented can be used to compute a-posteriori performance bounds for any \emph{fixed} external demand pattern efficiently, that is, by two forward simulations using BE and relaxed BE control respectively, without the need to solve an optimization problem. In addition, extensive evaluations suggest that the optimality gaps tend to be negligibly small. We further demonstrated that the best-effort controller is a suitable proxy for Alinea, which uses feedback to mitigate the lack of exact model knowledge. These considerations provide a theoretical explanation for part of the conclusions drawn from practical experience and heuristic considerations about performance of distributed ramp metering in comparison to the optimal soluton.

However, they seem to contradict other practical experience reported in \cite{papamichail2010coordinated}, which that suggests that``limited ramp storage space and the requirement of equity \dots are the main reasons for coordinated ramp metering". We can resolve this seeming contradiction by recalling that the optimality results apply to the monotonic CTM. In practice, there is empirical evidence of a capacity drop at a congested bottleneck. 
A congestion in such a model will reduce the bottleneck flow and subsequently decrease densities further downstream, which implies that the system dynamics are no longer monotonic. In a non-monotonic setting, an incentive exists to prevent congestion of a bottleneck, even if there is no danger of spill back of the congestion queue. An important conclusion about the potential benefits of coordinated ramp metering can be drawn from this analysis: coordinated ramp metering may target inefficiencies that result from limited space on the onramps \emph{in conjunction with the non-monotonic behavior of a congested bottleneck}. Any model-based, coordinated ramp metering policy should therefore either employ a model that is able to reproduce this effect, in order to recognize and avoid it, or include heuristic modifications to the controller targeting the unmodelled effect, such as heuristic ramp coordination rules. By contrast, model-predictive control based on the monotonic CTM with a maximize-TDT or minimize-TTS objective is unlikely to provide any substantial improvement over best-effort control (in the ideal case with perfect model knowledge) or Alinea (in a realistic setting), as there is no inherent incentive (for the controller based on the monotonic model) to coordinate ramps to avoid downstream congestion before spill-back occurs.

Based on these results, two main future research directions seem promising: First, an a-priori quantification of the suboptimality incurred because of infrequent violation of the optimality conditions or because of uncertainty in the knowledge of the fundamental diagram would be helpful. Theoretical bounds on the suboptimality in the presence of a general, not necessarily monotonic CTM would be very helpful, but since no convex reformulation or other efficient, global solution methods for optimal control problems involving the non-monotonic CTM exist, it is questionable, whether tight bounds can be derived. Instead, the known results on the impact of non-monotonic behavior can now be used to design control laws for non-monotonic, realistic freeway models in a more systematic manner by starting from the optimal solution for the monotonic case and specifically targeting the non-monotonicities with modifications to this baseline control policy.

\section*{Acknowledgements}
We are grateful to the Grenoble Traffic Lab at INRIA, Grenoble, France for providing the traffic data used for evaluation.


\FloatBarrier
\bibliography{trb_bibliography}

\FloatBarrier
\begin{appendix}

\section{Proof of Lemma \ref{prop:monotonic_in_inputs}} 
\label{proof:monotonic_in_flows}

\begin{proof} 
As in the proof of Lemma \ref{prop:monotonic_in_flows}, we make use of the fact that we can write the CCTM for $k \in \{1,\dots,n-1\}$ as
\begin{equation*}
\Phi_k(t+1) = \min \left\{ \Phi_k(t) + \Delta t~ d_k(\rho_k(t)), \Phi_k(t) + \Delta t~ F_k, \Phi_k(t) + \Delta t~ s_{k+1}(\rho_{k+1}(t)) \right\} .
\end{equation*}
The minimum function preserves monotonicity of its arguments. We can therefore verify monotonicity of the ACTM by checking that each of the functions:
\begin{align*}
f^{(-)}_k(\Phi_{k-1}(t), \Phi_k(t)) &:=  \Phi_k(t) + \Delta t~ d_k \left( \frac{1}{l_k} \left( \Phi_{k-1}(t) - \frac{1}{\bar\beta_k} \Phi_{k}(t) + R_k(t) \right) \right) , \\
f^{(+)}_k(\Phi_k(t), \Phi_{k+1}(t)) &:= \Phi_k(t) + \Delta t~ s_{k+1} \left( \frac{1}{l_{k+1}} \left( \Phi_{k}(t) - \frac{1}{\bar\beta_{k+1}} \Phi_{k+1}(t) + R_{k+1}(t) \right) \right)
\end{align*}
is nondecreasing in $R_k(t)$ and nonincreasing in $R_{k+1}(t)$.

\begin{itemize}
\item To verify that $\Phi_k(t+1)$ is nondecreasing in $R_k(t)$, first note that $f^{(+)}_k$ does not depend on $R_k(t)$ and is therefore trivially nondecreasing in this argument. 
Furthermore, $\rho_k(t)$ is nondecreasing in $R_{k}(t)$, since by assumption $l_k >0$. Also, $d_k(\cdot)$ is nondecreasing according to Assumption \ref{assumption:bounded_slope}. Recalling that $\Delta t > 0$, we conclude that $f^{(-)}_k$, which is a composition of the previously analyzed functions, is nondecreasing in $\R_{k}(t)$.
\item To verify that $\Phi_{k+1}(t)$ is nonincreasing in $R_{k+1}$, first note that $f^{(-)}_k$ does not depend on $R_{k+1}(t)$ and is therefore trivially nonincreasing in this argument. Furthermore, $\rho_{k+1}(t)$ is nondecreasing in $R_{k+1}(t)$, since by assumption $l_k >0$. Also, $s_{k+1}(\cdot)$ is nonincreasing according to Assumption \ref{assumption:bounded_slope}. Recalling that $\Delta t > 0$, we conclude that $f^{(+)}_k$, which is a composition of the previously analyzed functions, is nonincreasing in $R_{k+1}(t)$.
\end{itemize}

The proof of the cases $k=0$ and $k=n$ follows along the same lines.
\end{proof}

\section{Parameter choices}\label{sec:parameters} 

The sensor placement on the real freeway is described in \cite{de2015grenoble}. The topology of the freeway, in particular the lengths $l_k$ of cells of the CTM, can then easily be recovered from publicly available sources, e.g.\ Open Street Maps. It is important to note that although ramp metering will be installed in the future, it is not in place as of now (\today). We report the actual lengths of the onramps $\tilde l^{on}_k$. The ramps in cells $5$, $7$, $8$, $11$, $14$, $16$ and $19$ are used for ramp metering in this work. We chose to increase the onramp queue length used for ramp metering to $l^{on}_k = 400$m for these ramps, a value also used in \cite{gomes2006optimal} and \cite{muralidharan2012optimal}. This seems justified since installation of ramp metering should go along with the construction of sufficiently long queues to promise any benefits. Ramps that are not used for ramp metering are assigned $l^{on}_k = 0$m. The bounds on the onramp queue lengths are expressed using the dimensionless onramp capacity $\bar q_k$ in this work, defined as the maximal number of cars in the onramp queue. One has $\bar q_k = l^{on}_k / 8$m, which corresponds to a traffic jam density of $125$ (cars)/km (see below the definition of $\bar \rho_k$). The maximal onramp flow was chosen as $\bar r_k = 1800$ (cars)/h for every onramp, a standard choice corresponding to one car every two seconds.
We model the freeway dynamics with the CTM using a triangular fundamental diagram, which is fully characterized by the three parameters $v_k$, $\rho_k^c$ and $\bar \rho_k$. The free-flow velocity $v = 90$km/h is equal to the speed limit on the freeway\footnote{On the real freeway, infrastructure to allow for Variable Speed Limits (VSL) is already in place. So far, this capability is only used in case of high air pollution, when the speed limit is lowered to $70$km/h. In this work, we do not consider VSL. }. The traffic jam density is chosen as $250$ (cars)/km in every cell. This is a standard choice for a two-lane freeway, related to the average length of vehicles and typical distances between vehicles in standstill. The critical density was estimated from collected traffic data. Some minor parameter adaptations have been necessary to reproduce typical congestion patterns seen in reality: The maximal flow in cell 20 was reduced to $F_{20} = 4300$ (cars/h), such that this section becomes a bottleneck in rush-hour times, as observed in reality. Furthermore, traffic data from the end of the freeway (sections $18$, $19$, $20$ and $21$) are disturbed by what seems to be spill-back effects of the roads downstream of the considered freeway. Since we do not have data about downstream traffic conditions available, we resort to a reasonable estimate of the critical densities and use the value that was identified for cell $17$. The critical density for the very last cell is obtained as $\rho_{21}^c = \frac{3}{2} \cdot \rho_{17}^c$, since the freeway widens to three lanes at the end. As soon as $v_k$, $\rho_k^c$ and  $\bar \rho_k$ are chosen, additional parameters like $w_k$ and $F_k$ can be computed using the standard equations. We assume that the split ratios $\beta_k$ are constant in time. The values are again estimated from real traffic data. We use a discretization of $\Delta t = 15$s, in accordance with the data transmission interval of the sensors\footnote{The sensors count cars over intervals of $15$s and then submit the aggregated data at once.}. A summary of all data is given in Table \ref{tab:parameter_values}, the parameter values that have been modified, as described in the previous section, are highlighted by a star $^*$.
\begin{table}[tp] 
\caption{Parameter values of the Grenoble Freeway}
\begin{center}
\begin{tabular}{cccccccc}
\toprule 
Cell $k$ & $l_k$ (km) & $\tilde l^{on}_k$ (m) & $l^{on}_k$ (m) & $\bar q_k$ (cars) & $v_k$ (km/h) & $\rho_k^c$ (cars/km) & $\beta_k$ (\%)\\ [1ex] \hline
1 & 0.4 & 100 & 0 & 0 & 90 & 49.0 & 0\\ [0.7ex]
2 & 0.5 & 400 & 0 & 0 & 90 & 59.6 & 0\\ [0.7ex]
3 & 0.6 & - & - & - & 90 & 55.0 & 0\\ [0.7ex]
4 & 0.5 & - & - & - & 90 & 55.0 & 10 \\ [0.7ex]
5 & 0.5 & 120 & $400^*$ & 50 & 90 & 47.9 & 0 \\ [0.7ex]
6 & 0.7 & - & - & - & 90 & 52.0 & 18 \\ [0.7ex]
7 & 0.5 & 150 & 400$^*$ & 50 & 90 & 55.0 & 0 \\ [0.7ex]
8 & 0.5 & 120 & 400$^*$ & 50 & 90 & 51.1 & 0 \\ [0.7ex]
9 & 0.7 & - & - & - & 90 & 54.2 & 0 \\ [0.7ex]
10 & 1.3 & - & - & - & 90 & 48.0 & 11 \\ [0.7ex]
11 & 0.5 & 300 & 400$^*$ & 50 & 90 & 48.0 & 0 \\ [0.7ex]
12 & 0.5 & - & - & - & 90 & 51.6 & 0 \\ [0.7ex]
13 & 0.5 & - & - & - & 90 & 49.5 & 10 \\ [0.7ex]
14 & 0.5 & 200 & 400$^*$ & 50 & 90 & 54.7 & 0 \\ [0.7ex]
15 & 0.5 & - & - & - & 90 & 51.2 & 16 \\ [0.7ex]
16 & 0.5 & 200 & 400$^*$ & 50 & 90 & 51.2 & 0 \\ [0.7ex]
17 & 0.5 & - & - & - & 90 & 56.1 & 0 \\ [0.7ex]
18 & 0.5 & - & - & - & 90 & 56.1$^*$ & 10 \\ [0.7ex]
19 & 0.5 & 200 & 400$^*$ & 50 & 90 & 56.1$^*$ & 0 \\ [0.7ex]
20 & 0.5 & - & - & - & 90 & 56.1$^*$ & 8 \\ [0.7ex]
21 & 0.5 & 60 & 0 & 0 & 90 & 84.2$^*$ & 0 \\ [1.5ex]
\bottomrule
\end{tabular}
\end{center}
\label{tab:parameter_values}
\end{table}%

\end{appendix}

\end{document}